\documentclass[11pt]{article}

\usepackage{version}         

 \usepackage{color}

\usepackage{tabulary}
\usepackage{cite}

\usepackage{multirow}
\usepackage[OT1]{fontenc}
\usepackage{type1cm}
\usepackage{amssymb}
\usepackage{mathrsfs}
\usepackage[english]{babel}
\usepackage[latin1]{inputenc}
\usepackage[T1]{fontenc}
\usepackage{palatino}
\usepackage{amsfonts}
\usepackage{amsmath}
\usepackage{amssymb}
\usepackage{amsthm}
\usepackage{bbm}
\usepackage{extarrows}
\usepackage{graphicx}
\usepackage[usenames,dvipsnames]{xcolor}
\usepackage{wrapfig}
\usepackage{type1cm}
\usepackage[bf]{caption}
\usepackage{esint}
\usepackage{version}
\usepackage[colorlinks = true, citecolor = black, linkcolor = black, urlcolor = black, pdfstartview=FitH]{hyperref}
\usepackage{caption}
\usepackage{tikz}
\usepackage{bbding, wasysym}
\usepackage{enumitem}
\usepackage{float}

\usepackage[left=2.3cm,top=3cm,right=2.3cm]{geometry}
\allowdisplaybreaks[4]

\numberwithin{equation}{section}

\theoremstyle{plain}
\newtheorem{theorem}{Theorem}[section]

\newtheorem{corollary}[theorem]{Corollary}
\newtheorem{proposition}[theorem]{Proposition}
\newtheorem{lemma}[theorem]{Lemma}

\theoremstyle{remark}
\newtheorem{remark}[theorem]{Remark}

\newtheorem*{ack}{Acknowledgement}

\theoremstyle{definition}
\newtheorem{definition}[theorem]{Definition}

\newcommand{\R}{\mathbb{R}}
\newcommand{\C}{\mathbb{C}}

\newcommand{\N}{\mathbb{N}}
\renewcommand{\le}{\leqslant}
\renewcommand{\ge}{\geqslant}

\renewcommand{\P}{\mathbb{P}}

\newcommand{\cH}{\mathcal{H}}

\newcommand{\cL}{\mathcal{L}}

\newcommand{\cF}{\mathcal{F}}

\newcommand{\cP}{\mathcal{P}}

\newcommand{\Z}{\mathbb{Z}}

\newcommand{\spt}{\operatorname{spt}}

\renewcommand{\emptyset}{\varnothing}
\renewcommand{\epsilon}{\varepsilon}
\renewcommand{\phi}{\varphi}

\begin{document}

\title{Spectrality and supports of infinite convolutions in $\mathbb{R}^d$}

\author{Yao-Qiang Li}

\date{}

\maketitle

\renewcommand{\thefootnote}{}

\begin{abstract}
\noindent We study the spectrality of a class of infinite convolutions in $\R^d$, generalizing a result given by Li, Miao and Wang in 2022 from $\R$ to $\R^d$. This allows us to easily construct spectral measures with and without compact supports in $\R^d$, and motivates us to systematically study the supports of infinite convolutions. In particular, we give a sufficient and necessary condition for infinite convolutions to exist with compact supports, generalizing a related well-known result which is widely used. After giving strong relations between supports of infinite convolutions and sets of infinite sums, we study the closedness and fractal dimensions of infinite sums of union sets in order to deal with non-compact supports of infinite convolutions. As an application of these new tools, we deduce that there are spectral measures with and without compact supports of arbitrary Hausdorff and packing dimensions in $\R^d$, generalizing another result given by Li, Miao and Wang in 2022 from $\R$ to $\R^d$.
\end{abstract}

\section{Introduction}

\subsection{Spectrality of infinite convolutions}
\indent

Let $d\in\N$. A Borel probability measure $\mu$ on $\R^d$ is called a \textit{spectral measure} if there exists a countable set $\Lambda\subseteq\R^d$ such that the family of exponential functions
$$\big\{e^{-2\pi i<\lambda,\text{ }\cdot\text{ }>}:\lambda\in\Lambda\big\}$$
forms an orthonormal basis in $L^2(\mu)$. We call $\Lambda$ a \textit{spectrum} of $\mu$.

The existence of spectra of measures was initiated by Fuglede \cite{F74} in 1974. It is a basic question in harmonic analysis since the orthonormal basis consisting of exponential functions is used for Fourier series expansions
of functions \cite{S06}. Note that any compactly supported spectral measure must be of pure type: either discrete with finite support, singularly continuous, or absolutely continuous \cite{HLL13,LW06}. Since Jorgensen and Pedersen \cite{JP98} found the first singularly continuous spectral measure supported on a Cantor set in 1998, the spectrality of fractal measures has been widely studied until now (see \cite{AFL19,AHH19,AH14,AHL15,AHLi15,CLW21,D12,DFY21,DHL13,DHL14,DL17,DS15,DC21,DHL19,DL15,FHW18,FT24,LW22,LW23,L11,LMW22,LMW24,LMW24jfa,LW24,LL17,LZWC21,LDZ22,S19,Y22} and the references therein).

Use $\cP(\R^d)$ to denote the set of all Borel Probability measures on $\R^d$. Given $\mu_1,\mu_2,\cdots\in\cP(\R^d)$, if the finite convolution
$$\mu_1*\mu_2*\cdots*\mu_n$$
converges weakly to a Borel probability measure, we denote the weak limit measure by the infinite convolution
$$\mu_1*\mu_2*\mu_3*\cdots$$
and say that the infinite convolution exists.

Let $\delta_a$ denote the \textit{Dirac measure} concentrated at the point $a$, and for any non-empty finite set $A\subseteq\R^d$, define the \textit{uniform discrete measure} supported on $A$ by
$$\delta_A:=\frac{1}{\#A}\sum_{a\in A}\delta_a$$
where $\#$ denotes the cardinality of a set.

A square matrix is called \textit{expanding} if all eigenvalues have modulus strictly greater than $1$. Given a $d\times d$ expanding integer matrix $R$ and a non-empty finite set $B\subseteq\Z^d$ of integer vectors, we call $(R,B)$ an \textit{admissible pair} in $\R^d$ if there exists a finite set $L\subseteq\Z^d$ with $\#L=\#B$ such that the matrix
$$\Big[\frac{1}{\sqrt{\#B}}e^{-2\pi i<R^{-1}b,l>}\Big]_{b\in B,l\in L}$$
is unitary. To emphasize $L$, we also call $(R,B,L)$ a \textit{Hadamard triple} in $\R^d$.

It is known that if $\big\{(R_k,B_k)\big\}_{1\le k\le n}$ are admissible pairs, then the finite convolution
$$\delta_{R_1^{-1}B_1}*\delta_{R_1^{-1}R_2^{-1}B_2}*\cdots*\delta_{R_1^{-1}R_2^{-1}\cdots R_n^{-1}B_n}$$
is a spectral measure for each $n\in\N$. Since infinite convolutions generated by admissible pairs were raised by Strichartz \cite{S00} in 2000, the following question has received a lot of attention: \textit{Given a sequence of admissible pairs $\{(R_k,B_k)\}_{k\ge1}$, under what conditions is the infinite convolution
$$\delta_{R_1^{-1}B_1}*\delta_{R_1^{-1}R_2^{-1}B_2}*\delta_{R_1^{-1}R_2^{-1}R_3^{-1}B_3}*\cdots$$
a spectral measure?} See for examples \cite{AFL19,AH14,AHL15,DHL19,DL17,FHW18,FT24,LMW24,LMW24jfa,LW24,LDZ22} for affirmative results for this question.

As mentioned in \cite{LMW22}, the spectrality of infinite convolutions in $\R^d$ is very complicated. Most of the existing research on the spectrality of infinite convolutions has focused on $\R$ with compact supports. In \cite{LMW22}, after giving criteria for the weak convergence of infinite convolutions in $\R^d$, Li, Miao and Wang \cite[Theorem 1.4]{LMW22} focused on the spectrality of a class of infinite convolutions in $\R$, which may not be compactly supported. As the first main result in this paper, we generalize \cite[Theorem 1.4]{LMW22} from $\R$ to $\R^d$ in the following Theorem \ref{spectrality} by studying the spectrality of a class of infinite convolutions in $\R^d$, which may also not be compactly supported.

First we generalize the concept of a sequence of nearly consecutive digit sets \cite{LMW22} in $\R$ to the concept of a sequence of nearly $d$-th power lattices in $\R^d$. Given a sequence of positive integers $\{m_k\}_{k\ge1}$ and a sequence of $d\times d$ real matrices $\{R_k\}_{k\ge1}$, we call a sequence $\{B_k\}_{k\ge1}$ of subsets of $\R^d$ \textit{a sequence of nearly $d$-th power lattices} with respect to $\{m_k\}_{k\ge1}$ and $\{R_k\}_{k\ge1}$ if
$$B_k\equiv\{0,1,\cdots,m_k-1\}^d\quad(\text{mod }R_k\Z^d)\quad\text{for all }k\in\N$$
and
\begin{equation}\label{nearly}
\sum_{k=1}^\infty\frac{1}{m_k^d}\#(B_k\setminus\{0,1,\cdots,m_k-1\}^d)<\infty.
\end{equation}
Note that $\{B_k\}_{k\ge1}$ are subsets of $\R^d$ and may not be subsets of $[0,\infty)^d$.

To generalize \cite[Theorem 1.4]{LMW22} from $\R$ to $\R^d$, the first difficulty is to find suitable high-dimensional generalizations for the one-dimensional conditions $b_k\le N_k$ and $b_k\mid N_k$ for each $k\in\N$ in \cite[Theorem 1.4]{LMW22}. In Theorem \ref{spectrality}, we find that the conditions, $[-m_k,m_k]^d\subseteq R_k^T[-1,1]^d$ and all entries of the $d\times d$ matrix $R_k$ are multiples of $m_k$ for each $k\in\N$, are suitable, where $R_k^T$ denotes the transpose of $R_k$.

\begin{theorem}\label{spectrality}
Let $d\in\N$, $\{m_k\}_{k\ge1}$ be a sequence of integers no less than $2$, $\{R_k\}_{k\ge1}$ be a sequence of $d\times d$ invertible integer matrices, and $\{B_k\}_{k\ge1}$ be a sequence of nearly $d$-th power lattices with respect to $\{m_k\}_{k\ge1}$ and $\{R_k\}_{k\ge1}$. If for every $k\in\N$, all entries of $R_k$ are multiples of $m_k$ and $[-m_k,m_k]^d\subseteq R_k^T[-1,1]^d$, then the infinite convolution
$$\mu=\delta_{R_1^{-1}B_1}*\delta_{R_1^{-1}R_2^{-1}B_2}*\delta_{R_1^{-1}R_2^{-1}R_3^{-1}B_3}*\cdots$$
exists and is a spectral measure with a spectrum in $\Z^d$.
\end{theorem}

Using this result, one can easily construct spectral measures with and without compact supports in $\R^d$.

To show the spectrality of $\mu$ in Theorem \ref{spectrality}, using Theorem \ref{cite}, a tool developed by Li and Wang \cite{LW24} recently, the main we need to prove is Lemma \ref{equi-lemma}. The key in the proof of Lemma \ref{equi-lemma} is to estimate the lower bound of the modulus of the Fourier transform of the push-forward measure of the tail of the infinite convolution $\mu$ on $[-2/3,2/3]^d$. One will see that the estimation for our high-dimensional case is much more intricate than the one-dimensional case given in the proof of \cite[Theorem 1.4]{LMW22}.

In the proof of the spectrality of $\mu$ in Theorem \ref{spectrality}, except for Lemma \ref{equi-lemma}, we establish Proposition \ref{implies expanding} to guarantee that $R_k$ is expanding and then Proposition \ref{mod HT} to guarantee that $(R_k,B_k)$ is an admissible pair for each $k\in\N$. In the proof of the existence of the infinite convolution $\mu$ in Theorem \ref{spectrality}, with the conditions (\ref{nearly}) and $[-m_k,m_k]^d\subseteq R_k^T[-1,1]^d$ for all $k\in\N$, except for using Corollary \ref{exist-cor}, we also need Proposition \ref{R to RT}. These disclose the complicacy of the high-dimensional case.

By Theorem \ref{spectrality} we immediately get the following, which generalizes \cite[Theorem 1.4]{AH14} from $\R$ to $\R^d$.

\begin{corollary}\label{spectrality-cor}
Let $d\in\N$. For every $k\in\N$, let $m_k\ge2$ be an integer, $B_k=\{0,1,\cdots,m_k-1\}^d$ and $R_k$ be a $d\times d$ invertible integer matrix such that all entries are multiples of $m_k$ and $[-m_k,m_k]^d\subseteq R_k^T[-1,1]^d$. Then the infinite convolution
$$\mu=\delta_{R_1^{-1}B_1}*\delta_{R_1^{-1}R_2^{-1}B_2}*\delta_{R_1^{-1}R_2^{-1}R_3^{-1}B_3}*\cdots$$
exists and is a spectral measure with a spectrum in $\Z^d$.
\end{corollary}

Note that the condition $[-m_k,m_k]^d\subseteq R_k^T[-1,1]^d$ in Theorem \ref{spectrality} and Corollary \ref{spectrality-cor} is not equivalent to $[-m_k,m_k]^d$ $\subseteq$ $R_k[-1,1]^d$. For example, take $d=2$, $m_k=m=2$ and
$$R_k=R=\left(\begin{matrix} 4 & -2 \\ 0 & 2 \end{matrix}\right)$$
for all $k\in\N$. Then
$$R^{-1}=\left(\begin{matrix}
1/4 & 1/4 \\
0 & 1/2
\end{matrix}\right)\quad\text{and}\quad(R^T)^{-1}=\left(\begin{matrix}
1/4 & 0 \\
1/4 & 1/2
\end{matrix}\right).$$
Since
$$R^{-1}\Big\{\left(\begin{matrix}2\\2\end{matrix}\right),\left(\begin{matrix}2\\-2\end{matrix}\right),\left(\begin{matrix}-2\\2\end{matrix}\right),\left(\begin{matrix}-2\\-2\end{matrix}\right)\Big\}\subseteq[-1,1]^2,$$
by the linearity of $R^{-1}$, we get $R^{-1}[-2,2]^2\subseteq[-1,1]^2$ and then $[-2,2]^2\subseteq R[-1,1]^2$. But
$$(R^T)^{-1}\left(\begin{matrix}2\\2\end{matrix}\right)=\left(\begin{matrix}1/2\\3/2\end{matrix}\right)\notin[-1,1]^2$$
implies $(R^T)^{-1}[-2,2]^2$ $\nsubseteq$ $[-1,1]^2$ and then $[-2,2]^2$ $\nsubseteq$ $R^T[-1,1]^2$.

\subsection{Supports of infinite convolutions}
\indent

Spectrality of measures is very important in harmonic analysis since the orthonormal basis consisting of exponential functions is used for Fourier series expansions of functions. The supports of spectral measures are also very important since they are actually related to the domains of the expanded functions.

In the last subsection, we confirm the spectrality of a class of infinite convolutions in $\R^d$, which may not be compactly supported. This motivates us to systematically study the supports of infinite convolutions in $\R^d$ in this subsection.

For any sets $A_1,A_2,\cdots\subseteq\R^d$, define the set of infinite sum
$$\sum_{k=1}^\infty A_k:=\Big\{x\in\R^d:\exists a_k\in A_k\text{ for each }k\in\N\text{ s.t. }x=\sum_{k=1}^\infty a_k\Big\}.$$
Use $\spt\mu$ to denote the support of the measure $\mu$, i.e., the smallest closed set with full measure, and use $\overline{A}$ to denote the closure of the set $A$, i.e., the smallest closed set containing $A$.

It is well-known and widely used \cite{AFL19,AHH19,AH14,AHL15,LMW22,LMW24,LDZ22,S06,WLD20,WX24,Y22} that: for any non-empty finite sets $A_1,A_2,\cdots\subseteq\R^d$, if
\begin{equation}\label{<infty}
\sum_{k=1}^\infty\max_{a\in A_k}|a|<\infty,
\end{equation}
then
\begin{equation}\label{exist-compact}
\delta_{A_1}*\delta_{A_2}*\cdots\text{ exists with compact support},
\end{equation}
and
\begin{equation}\label{spt=sum}
\spt(\delta_{A_1}*\delta_{A_2}*\cdots)=\sum_{k=1}^\infty A_k,
\end{equation}
noting that $\sum_{k=1}^\infty a_k$ converges for any $a_k\in A_k$ with $k\in\N$ by (\ref{<infty}). Thus a lot of existing research on the spectrality of infinite convolutions with compact supports has assumed (\ref{<infty}), which is a sufficient but not necessary condition for (\ref{exist-compact}). We generalize this well-known result by showing in Corollary \ref{spt-cor} (2) that a sufficient and necessary condition for (\ref{exist-compact}) is (\ref{min-max con}), which still implies (\ref{spt=sum}). In general when $\delta_{A_1}*\delta_{A_2}*\cdots$ may not be compactly supported, Remark \ref{spt-re} says that $\sum_{k=1}^\infty A_k$ may not be closed and Corollary \ref{spt-cor} (1) gives
$$\spt(\delta_{A_1}*\delta_{A_2}*\cdots)=\overline{\sum_{k=1}^\infty A_k}.$$
We will see that this is a powerful tool to study $\spt(\delta_{A_1}*\delta_{A_2}*\cdots)$ even if it is not compact.

First we fix some notations and terminologies. For all $A\subseteq\R^d$ and $j\in\{1,\cdots,d\}$, denote the $j$-th coordinate projection of $A$ from $\R^d$ to $\R$ by
$$(A)_j:=\big\{a_j\in\R:(a_1,\cdots,a_j,\cdots,a_d)\in A\big\}.$$
For a non-empty set $A\subseteq\R$, we use $\inf A\in\R\cup\{-\infty\}$ and $\sup A\in\R\cup\{+\infty\}$ respectively to denote the infimum and supremum of $A$, and use $\min A\in\R$ and $\max A\in\R$ respectively to denote the minimal and maximal of $A$ if they exist. Given $x_1,x_2,\cdots\in\R\cup\{\pm\infty\}$, we say that $\sum_{k=1}^\infty x_k$ converges if $x_k\in\R$ (not $\pm\infty$) for all $k\in\N$ and the limit $\lim_{n\to\infty}\sum_{k=1}^nx_k$ exists (not $\pm\infty$). Otherwise we say that $\sum_{k=1}^\infty x_k$ diverges.

Now we give the following theorem as the second main result in this paper, from which Corollary \ref{spt-cor} can be deduced.

\begin{theorem}\label{spt}
Let $d\in\N$ and $\mu_1,\mu_2,\cdots\in\cP(\R^d)$.
\vspace{5pt}
\newline\emph{(1)} If the infinite convolution $\mu_1*\mu_2*\cdots$ exists, then
$$\spt(\mu_1*\mu_2*\cdots)=\overline{\sum_{k=1}^\infty\spt\mu_k}.$$
\emph{(2)} The infinite convolution $\mu_1*\mu_2*\cdots$ exists with compact support if and only if
\begin{equation}\label{spt inf-sup con}
\sum_{k=1}^\infty\inf(\spt\mu_k)_j\text{ and }\sum_{k=1}^\infty\sup(\spt\mu_k)_j\text{ converge}\quad\text{for all }j\in\{1,\cdots,d\},
\end{equation}
which implies
$$\spt(\mu_1*\mu_2*\cdots)=\sum_{k=1}^\infty\spt\mu_k.$$
\end{theorem}

To prove Theorem \ref{spt}, we need to use Theorem \ref{old}, probability theory, Propositions \ref{+=} and \ref{bounded-closed}, where the proof of Proposition \ref{bounded-closed} is based on some detailed convergence analysis and a translation technique for sets of infinite sums.

By Theorem \ref{spt} we immediately get the following.

\begin{corollary}\label{spt-cor}
Let $d\in\N$ and $A_1,A_2,\cdots\subseteq\R^d$ be non-empty finite sets.
\vspace{5pt}
\newline\emph{(1)} If the infinite convolution $\delta_{A_1}*\delta_{A_2}*\cdots$ exists, then
$$\spt(\delta_{A_1}*\delta_{A_2}*\cdots)=\overline{\sum_{k=1}^\infty A_k}.$$
\emph{(2)} The infinite convolution $\delta_{A_1}*\delta_{A_2}*\cdots$ exists with compact support if and only if
\begin{equation}\label{min-max con}
\sum_{k=1}^\infty\min(A_k)_j\text{ and }\sum_{k=1}^\infty\max(A_k)_j\text{ converge}\quad\text{for all }j\in\{1,\cdots,d\},
\end{equation}
which implies
$$\spt(\delta_{A_1}*\delta_{A_2}*\cdots)=\sum_{k=1}^\infty A_k.$$
\end{corollary}

Note that (\ref{<infty}) implies (\ref{min-max con}) since
$$\max\Big\{\sum_{k=1}^\infty\big|\min(A_k)_j\big|,\sum_{k=1}^\infty\big|\max(A_k)_j\big|\Big\}\le\sum_{k=1}^\infty\max_{(a_1,\cdots,a_d)\in A_k}|a_j|\le\sum_{k=1}^\infty\max_{a\in A_k}|a|.$$

\begin{remark}\label{spt-re}
In Corollary \ref{spt-cor} (1), $\overline{\sum_{k=1}^\infty A_k}$ can not be simplified to $\sum_{k=1}^\infty A_k$. At the end of Section 4, we will give an example in which $A_1,A_2,\cdots\subseteq[0,1]$ are non-empty finite sets such that $\delta_{A_1}*\delta_{A_2}*\cdots$ exists but $\sum_{k=1}^\infty A_k$ is not closed.
\end{remark}

For the one-dimensional case, we will deduce the following in addition.

\begin{corollary}\label{spt-1}
Let $\mu_1,\mu_2,\cdots\in\cP(\R)$ with compact supports such that $\mu_1*\mu_2*\cdots$ exists and
$$\lim_{k\to\infty}\big(\max(\spt\mu_k)-\min(\spt\mu_k)\big)=0.$$
\emph{(1)} If $\sum_{k=1}^\infty\min(\spt\mu_k)$ converges and $\sum_{k=1}^\infty\max(\spt\mu_k)$ diverges, then
$$\spt(\mu_1*\mu_2*\cdots)=\Big[\sum_{k=1}^\infty\min(\spt\mu_k),+\infty\Big).$$
\emph{(2)} If $\sum_{k=1}^\infty\min(\spt\mu_k)$ diverges and $\sum_{k=1}^\infty\max(\spt\mu_k)$ converges, then
$$\spt(\mu_1*\mu_2*\cdots)=\Big(-\infty,\sum_{k=1}^\infty\max(\spt\mu_k)\Big].$$
\emph{(3)} If both $\sum_{k=1}^\infty\min(\spt\mu_k)$ and $\sum_{k=1}^\infty\max(\spt\mu_k)$ diverge, then
$$\spt(\mu_1*\mu_2*\cdots)=\R.$$
\end{corollary}

Immediately we get the following.

\begin{corollary}\label{spt-1-cor}
Let $A_1,A_2,\cdots\subseteq\R$ be non-empty finite sets such that $\delta_{A_1}*\delta_{A_2}*\cdots$ exists and
$$\lim_{k\to\infty}(\max A_k-\min A_k)=0.$$
\emph{(1)} If $\sum_{k=1}^\infty\min A_k$ converges and $\sum_{k=1}^\infty\max A_k$ diverges, then
$$\spt(\delta_{A_1}*\delta_{A_2}*\cdots)=\Big[\sum_{k=1}^\infty\min A_k,+\infty\Big).$$
\emph{(2)} If $\sum_{k=1}^\infty\min A_k$ diverges and $\sum_{k=1}^\infty\max A_k$ converges, then
$$\spt(\delta_{A_1}*\delta_{A_2}*\cdots)=\Big(-\infty,\sum_{k=1}^\infty\max A_k\Big].$$
\emph{(3)} If both $\sum_{k=1}^\infty\min A_k$ and $\sum_{k=1}^\infty\max A_k$ diverge, then
$$\spt(\delta_{A_1}*\delta_{A_2}*\cdots)=\R.$$
\end{corollary}

All cases in Corollaries \ref{spt-1} and \ref{spt-1-cor} are possible. We give examples for Corollary \ref{spt-1-cor} in the following.
\begin{itemize}
\item[\emph{(1)}] For all $k\in\N$, let $A_k:=\{0,\frac{1}{k^3},\frac{2}{k^3},\cdots,\frac{k-1}{k^3},\frac{1}{k}\}$. Then $\lim_{k\to\infty}(\max A_k-\min A_k)=\lim_{k\to\infty}\frac{1}{k}=0$, $\sum_{k=1}^\infty\min A_k=0$ converges, $\sum_{k=1}^\infty\max A_k=\sum_{k=1}^\infty\frac{1}{k}$ diverges, and it follows from Corollary \ref{exist-cor} and
    $$\sum_{k=1}^\infty\frac{1}{\#A_k}\sum_{a\in A_k}\frac{|a|}{1+|a|}\le\sum_{k=1}^\infty\frac{1}{\#A_k}\sum_{a\in A_k}|a|=\sum_{k=1}^\infty\frac{1}{k+1}\Big(\frac{k-1}{2k^2}+\frac{1}{k}\Big)<\infty$$
    that $\delta_{A_1}*\delta_{A_2}*\cdots$ exists. By Corollary \ref{spt-1-cor} (1) we get $\spt(\delta_{A_1}*\delta_{A_2}*\cdots)=[0,+\infty)$.
\item[\emph{(2)}] For all $k\in\N$, let $A_k:=\{0,-\frac{1}{k^3},-\frac{2}{k^3},\cdots,-\frac{k-1}{k^3},-\frac{1}{k}\}$. Then $\lim_{k\to\infty}(\max A_k-\min A_k)=\lim_{k\to\infty}\frac{1}{k}=0$, $\sum_{k=1}^\infty\min A_k=\sum_{k=1}^\infty(-\frac{1}{k})$ diverges, $\sum_{k=1}^\infty\max A_k=0$ converges, and it follows from the same way as the above (1) that $\delta_{A_1}*\delta_{A_2}*\cdots$ exists. By Corollary \ref{spt-1-cor} (2) we get $\spt(\delta_{A_1}*\delta_{A_2}*\cdots)=(-\infty,0]$.
\item[\emph{(3)}] For all $k\in\N$, let $A_k:=\{0,\pm\frac{1}{k^3},\pm\frac{2}{k^3},\cdots,\pm\frac{k-1}{k^3},\pm\frac{1}{k}\}$. Then $\lim_{k\to\infty}(\max A_k-\min A_k)=\lim_{k\to\infty}\frac{2}{k}=0$, both $\sum_{k=1}^\infty\min A_k=\sum_{k=1}^\infty(-\frac{1}{k})$ and $\sum_{k=1}^\infty\max A_k=\sum_{k=1}^\infty\frac{1}{k}$ diverge, and it follows in a way similar to the above (1) that $\delta_{A_1}*\delta_{A_2}*\cdots$ exists. By Corollary \ref{spt-1-cor} (3) we get $\spt(\delta_{A_1}*\delta_{A_2}*\cdots)=\R$.
\end{itemize}

\subsection{Infinite sums of union sets}
\indent

In Theorem \ref{spt} (2), we know that when $\mu_1*\mu_2*\cdots$ exists with compact support, we have $\spt(\mu_1*\mu_2*\cdots)=\sum_{k=1}^\infty\spt\mu_k$. But in general, when $\spt(\mu_1*\mu_2*\cdots)$ is not necessarily compact, Remark \ref{spt-re} tells us that $\sum_{k=1}^\infty\spt\mu_k$ may not be closed, and we can only get $\spt(\mu_1*\mu_2*\cdots)=\overline{\sum_{k=1}^\infty\spt\mu_k}$ from Theorem \ref{spt} (1). In order to obtain $\spt(\mu_1*\mu_2*\cdots)=\sum_{k=1}^\infty\spt\mu_k$ to provide convenience for further research on $\spt(\mu_1*\mu_2*\cdots)$, we should study under what conditions $\sum_{k=1}^\infty\spt\mu_k$ is a closed set, especially when $\spt(\mu_1*\mu_2*\cdots)$ is not compact and equivalently $\sum_{k=1}^\infty\spt\mu_k$ is not bounded. Therefore, in this subsection we study infinite sums of union sets of the form $\sum_{k=1}^\infty(A_k\cup A_k')$, where $\{A_k'\}_{k\ge1}$ are far from the original point in some sense, including the conditions under which $\sum_{k=1}^\infty(A_k\cup A_k')$ is closed and the fractal dimensions of $\sum_{k=1}^\infty(A_k\cup A_k')$. These tools will be applied to the non-compact supports of spectral measures in the next subsection.

For $s\in[0,d]$, use $\cH^s$ and $\cP^s$ to denote the $s$-dimensional Hausdorff measure and packing measure respectively \cite{F14}. Besides, we use $\dim_H$ and $\dim_P$ to denote the Hausdorff dimension and packing dimension respectively. Regard $\inf\emptyset=\min\emptyset=+\infty$ and $\sup\emptyset=\max\emptyset=-\infty$ throughout this paper.

As the third main result in this paper, the following Theorem \ref{sum union} on the one hand gives some relatively weak conditions for $\sum_{k=1}^\infty(A_k\cup A_k')$ to be closed, and on the other hand provides a way to simplify the calculation for the Hausdorff and packing dimensions of $\sum_{k=1}^\infty(A_k\cup A_k')$ by only considering $\sum_{k=1}^\infty A_k$, which allows us to use classical results on fractal dimensions for bounded sets (as we will see in the proof of Corollary \ref{dim-sum}) instead of dealing with unbounded sets under certain circumstances.

\begin{theorem}\label{sum union}
For each $k\in\N$, let $A_k,A_k'\subseteq\R^d$ where $A_k$ is non-empty and $A_k'$ may be empty.
\vspace{5pt}
\newline\emph{(1)} Suppose that $A_k\cup A_k'$ is closed for every $k\in\N$, and for every $j\in\{1,\cdots,d\}$
\begin{equation}\label{inf-sup con}
\text{both }\sum_{k=1}^\infty\inf(A_k)_j\text{ and }\sum_{k=1}^\infty\sup(A_k)_j\text{ converge}
\end{equation}
and
\begin{equation}\label{inf-sup or}
\inf_{k\in\N}\inf(A_k')_j>-\infty\quad\text{or}\quad\sup_{k\in\N}\sup(A_k')_j<+\infty.
\end{equation}
If $\lim_{k\to\infty}\inf_{a\in A_k'}|a|=+\infty$, then
$$\sum_{k=1}^\infty\big(A_k\cup A_k'\big)\text{ is closed.}$$
\emph{(2)} Suppose that $A_k'$ is at most countable for every $k\in\N$. If $\varliminf_{k\to\infty}\inf_{a\in A_k'}|a|>0$, then
\begin{itemize}
\item[\textcircled{\footnotesize{$1$}}] $\cH^s\big(\sum_{k=1}^\infty(A_k\cup A_k')\big)=0$ if and only if $\cH^s\big(\sum_{k=1}^\infty A_k\big)=0$ for all $s\in[0,d]$, and
    $$\dim_H\sum_{k=1}^\infty\big(A_k\cup A_k'\big)=\dim_H\sum_{k=1}^\infty A_k;$$
\item[\textcircled{\footnotesize{$2$}}] $\cP^s\big(\sum_{k=1}^\infty(A_k\cup A_k')\big)=0$ if and only if $\cP^s\big(\sum_{k=1}^\infty A_k\big)=0$ for all $s\in[0,d]$, and
    $$\dim_P\sum_{k=1}^\infty\big(A_k\cup A_k'\big)=\dim_P\sum_{k=1}^\infty A_k.$$
\end{itemize}
\end{theorem}

In Theorem \ref{sum union} (1), the condition that (\ref{inf-sup or}) holds for every $j\in\{1,\cdots,d\}$ is equivalent to
$$\exists c_1,c_2\in\R\text{ and }X_1,\cdots,X_d\in\Big\{[c_1,+\infty),(-\infty,c_2]\Big\}\text{ s.t. }\bigcup_{k=1}^\infty A_k'\subseteq X_1\times\cdots\times X_d.$$
This is not difficult to satisfy. For example
$$\bigcup_{k=1}^\infty A_k'\subseteq[-c,\infty)^d\quad\text{or}\quad\bigcup_{k=1}^\infty A_k'\subseteq(-\infty,c]^d\quad\text{for some }c\ge0.$$

The proof of Theorem \ref{sum union} relies on the decomposition of $\sum_{k=1}^\infty(A_k\cup A_k')$ in Propositions \ref{decom}. Besides, a translation technique similar to the one in the proof of Proposition \ref{bounded-closed} (2) and technical estimations on the absolute values of specific sums of the coordinate components of certain summable points play a key role in the proof of Theorem \ref{sum union} (1).

Every condition in Theorem \ref{sum union} can not be omitted. Otherwise one can construct examples such that the conclusions may not hold. Write $S:=\sum_{k=1}^\infty(A_k\cup A_k')$. We note the following for Theorem \ref{sum union} (1).
\begin{itemize}
\item[\textcircled{\footnotesize{$1$}}] The condition that $\sum_{k=1}^\infty\sup(A_k)_j$ converges for every $j\in\{1,\cdots,d\}$ can not be omitted. Otherwise, we can take $d=1$, $A_k=\big\{0,\frac{k}{k+1}\big\}$ and $A_k'=\big\{k+1\big\}$ (or $\emptyset$) for all $k\in\N$. Then $\sum_{k=1}^\infty\inf A_k=0$ converges, $\inf_{k\in\N}\inf A_k'>-\infty$, and $\lim_{k\to\infty}\inf_{a\in A_k'}|a|=+\infty$. But $S$ is not closed, since one can easily verify $1\in\overline{S}\setminus S$.
\item[\textcircled{\footnotesize{$2$}}] The condition that (\ref{inf-sup or}) holds for every $j\in\{1,\cdots,d\}$ can not be omitted. Otherwise, we can take $d=1$, $A_k=\big\{0\big\}$ and $A_k'=\big\{(-1)^k(k+\frac{1}{2^k})\big\}$ for all $k\in\N$. Then (\ref{inf-sup con}) holds for $j=1$, and $\lim_{k\to\infty}\inf_{a\in A_k'}|a|=+\infty$. But $S$ is not closed, since one can easily verify $1\in\overline{S}\setminus S$.
\item[\textcircled{\footnotesize{$3$}}] The condition $\lim_{k\to\infty}\inf_{a\in A_k'}|a|=+\infty$ can not be weakened to $\varlimsup_{k\to\infty}\inf_{a\in A_k'}|a|=+\infty$ and $\varliminf_{k\to\infty}\inf_{a\in A_k'}|a|>0$. Otherwise, we can take $d=1$, $A_k=\big\{0\big\}$ for all $k\in\N$, $A_k'=\big\{\frac{k}{k+1}\big\}$ for all odd $k\in\N$ and $A_k'=\big\{k\big\}$ for all even $k\in\N$. Then (\ref{inf-sup con}) and (\ref{inf-sup or}) hold for $j=1$, $\varlimsup_{k\to\infty}\inf_{a\in A_k'}|a|=+\infty$ and $\varliminf_{k\to\infty}\inf_{a\in A_k'}|a|=1>0$. But $S$ is not closed, since one can easily verify $1\in\overline{S}\setminus S$.
\end{itemize}

In the following corollary, which will be applied in the next subsection, we give results on a special class of infinite sums which may not be bounded, including conditions under which they are closed and concrete formulae for their Hausdorff and packing dimensions. Here we use $\cL^d$ to denote the $d$-dimensional Lebesgue measure.

\begin{corollary}\label{dim-sum} Let $d\in\N$. For each $k\in\N$, let $c_k\ge1$ and $C_k\ge c_k+1$ be real numbers, and $B_k\subseteq\R^d$ with $\emptyset\neq B_k\cap[0,c_k]^d\subseteq\Z^d$.
\vspace{5pt}
\newline\emph{(1)} Suppose that $B_1,B_2,\cdots$ are all closed, and for every $j\in\{1,\cdots,d\}$,
\begin{equation}\label{cor inf-sup or}
\inf_{k\in\N}\big(C_1^{-1}\cdots C_k^{-1}\inf(B_k)_j\big)>-\infty\quad\text{or}\quad\sup_{k\in\N}\big(C_1^{-1}\cdots C_k^{-1}\sup(B_k)_j\big)<+\infty.
\end{equation}
If
$$\lim_{k\to\infty}\frac{\inf\big\{|x|:x\in B_k\setminus[0,c_k]^d\big\}}{C_1\cdots C_k}=+\infty,$$
then
$$\sum_{k=1}^\infty C_1^{-1}\cdots C_k^{-1}B_k\text{ is closed.}$$
\emph{(2)} Suppose that $B_k$ is at most countable for every $k\in\N$ and
$$\varliminf_{k\to\infty}\frac{\inf\big\{|x|:x\in B_k\setminus[0,c_k]^d\big\}}{C_1\cdots C_k}>0.$$
\begin{itemize}
\item[\textcircled{\footnotesize{$1$}}] If $\prod_{k=1}^\infty\frac{\#(B_k\cap[0,c_k]^d)}{C_k^d}=0$, then $\cL^d(\sum_{k=1}^\infty C_1^{-1}\cdots C_k^{-1}B_k)=0$.
\item[\textcircled{\footnotesize{$2$}}] If $\lim_{k\to\infty}\frac{\log C_k}{\log C_1\cdots C_k}=0$, then
$$\dim_H\sum_{k=1}^\infty C_1^{-1}\cdots C_k^{-1}B_k=\varliminf_{k\to\infty}\frac{\log\#(B_1\cap[0,c_1]^d)\cdots\#(B_k\cap[0,c_k]^d)}{\log C_1\cdots C_k},$$
and
$$\dim_P\sum_{k=1}^\infty C_1^{-1}\cdots C_k^{-1}B_k=\varlimsup_{k\to\infty}\frac{\log\#(B_1\cap[0,c_1]^d)\cdots\#(B_k\cap[0,c_k]^d)}{\log C_1\cdots C_k}.$$
\end{itemize}
\end{corollary}

\subsection{Spectral measures with and without compact supports of arbitrary dimensions}
\indent

In \cite[Theorem 1.7]{LMW22}, Li, Miao and Wang showed that there are spectral measures without compact supports of arbitrary Hausdorff and packing dimensions in $\R$. To get this result, they used \cite[Theorem 1.4]{LMW22} to construct a special class of spectral measures with the form of infinite convolutions in $\R$, proved that the supports of these infinite convolutions are countable unions of specific compact sets \cite[Proposition 5.1]{LMW22}, and then used these specific compact sets in the proof of the Hausdorff and packing dimension formulae for the supports of the corresponding infinite convolutions \cite[Proposition 5.3]{LMW22}.

Different from their ideas, after using Theorem \ref{spectrality} to construct spectral measures with the form of infinite convolutions in $\R^d$, we systematically study the supports of general infinite convolutions in Subsection 1.2 and the infinite sums of union sets in Subsection 1.3. Finally in this subsection, as an application of the tools developed in the above subsections, we deduce that there are spectral measures with and without compact supports of arbitrary Hausdorff and packing dimensions in $\R^d$, generalizing \cite[Theorem 1.7]{LMW22} from $\R$ to $\R^d$ in a different way.

First, taking $c_k:=m_k-1$ and $C_k:=N_k$ for all $k\in\N$ in Corollary \ref{dim-sum}, by Theorem \ref{spectrality} and Corollary \ref{spt-cor} (1) we can get the following immediately.

\begin{corollary}\label{dim-spt}
Let $d\in\N$. For each $k\in\N$, let $N_k\ge m_k\ge2$ be integers with $m_k\mid N_k$ and $B_k\subseteq\{0,1,2,\cdots\}^d$ be a finite set with $G_k:=B_k\cap\{0,1,\cdots,m_k-1\}^d\neq\emptyset$. Suppose
\begin{equation}\label{cor lim infty}
\lim_{k\to\infty}\frac{\min\{|x|:x\in B_k\setminus\{0,1,\cdots,m_k-1\}^d\}}{N_1\cdots N_k}=+\infty,
\end{equation}
$$\lim_{k\to\infty}\frac{\log N_k}{\log N_1\cdots N_k}=0\quad\text{and}\quad\prod_{k=1}^\infty\frac{\#G_k}{N_k^d}=0,$$
and suppose that $\{B_k\}_{k\ge1}$ is a sequence of nearly $d$-th power lattices with respect to $\{m_k\}_{k\ge1}$ and the sequence of $d\times d$ diagonal matrices $\{\text{diag}(N_k,\cdots,N_k)\}_{k\ge1}$. Then the infinite convolution
$$\mu=\delta_{N_1^{-1}B_1}*\delta_{N_1^{-1}N_2^{-1}B_2}*\delta_{N_1^{-1}N_2^{-1}N_3^{-1}B_3}*\cdots$$
exists, is a singular spectral measure with a spectrum in $\Z^d$, $\spt\mu=\sum_{k=1}^\infty N_1^{-1}\cdots N_k^{-1}B_k$,
$$\dim_H\spt\mu=\varliminf_{k\to\infty}\frac{\log\#G_1\cdots\#G_k}{\log N_1\cdots N_k}\quad\text{and}\quad\dim_P\spt\mu=\varlimsup_{k\to\infty}\frac{\log\#G_1\cdots\#G_k}{\log N_1\cdots N_k}.$$
\end{corollary}

For spectral measures with compact supports in $\R^d$, we have the following. The similar result for spectral measures on $\R$ can be found in \cite{DS15}.

\begin{corollary}\label{inter-com}
Let $d\in\N$. For any $\alpha,\beta\in[0,d]$ with $\alpha\le\beta$, there exists a singular spectral measure $\mu$ on $\R^d$ with a spectrum in $\Z^d$ and with compact support such that
$$\dim_H\spt\mu=\alpha\quad\text{and}\quad\dim_P\spt\mu=\beta.$$
\end{corollary}

Finally we consider spectral measures without compact supports and generalize \cite[Theorem 1.7]{LMW22} to $\R^d$.

\begin{corollary}\label{inter-non}
Let $d\in\N$. For any $\alpha,\beta\in[0,d]$ with $\alpha\le\beta$, there exists a singular spectral measure $\mu$ on $\R^d$ with a spectrum in $\Z^d$ and without compact support such that
$$\dim_H\spt\mu=\alpha\quad\text{and}\quad\dim_P\spt\mu=\beta.$$
\end{corollary}

This paper is organized as follows. In the next section we give some preliminaries. Then we prove Theorem \ref{spectrality} in Section 3, prove Theorem \ref{spt}, Corollary \ref{spt-1} and give an example for Remark \ref{spt-re} in Section 4, prove Theorem \ref{sum union} and Corollary \ref{dim-sum} in Section 5, and finally deduce Corollaries \ref{inter-com} and \ref{inter-non} in Section 6.

\section{Preliminaries}
\indent

Recall that $\cP(\R^d)$ denotes the set of all Borel probability measures on $\R^d$. For $\mu\in\cP(\R^d)$ the \textit{Fourier transform} of $\mu$ is defined by
$$\widehat{\mu}(\xi):=\int_{\R^d}e^{-2\pi i<\xi,x>}d\mu(x)\quad\text{for all }\xi\in\R^d.$$
It is well-known that $\widehat{\mu}$ is a bounded, continuous function with $\widehat{\mu}(\mathbf{0})=1$. See for example \cite[Theorem 13.1]{JP03}.

For $\mu,\mu_1,\mu_2,\cdots\in\cP(\R^d)$, we say that $\mu_n$ \textit{converges weakly} to $\mu$ if
$$\int_{\R^d}f\text{ }d\mu_n\to\int_{\R^d}f\text{ }d\mu\quad\text{for all }f\in C_b(\R^d),$$
where $C_b(\R^d)$ denotes the set of all bounded continuous functions on $\R^d$.

For $\mu,\nu\in\cP(\R^d)$, the \textit{convolution} $\mu*\nu$ is defined by
$$\mu*\nu(B):=\int_{\R^d}\mu(B-y)\text{ }d\nu(y)=\int_{\R^d}\nu(B-x)\text{ }d\mu(x)\quad\text{for every Borel set }B\subseteq\R^d.$$
Equivalently, $\mu*\nu$ is the unique Borel probability measure on $\R^d$ satisfying
$$\int_{\R^d}f(x)\text{ }d\mu*\nu(x)=\int_{\R^d\times\R^d}f(x+y)\text{ }d\mu\times\nu(x,y)\quad\text{for all }f\in C_b(\R^d).$$
It is straightforward to see $\widehat{\mu*\nu}(\xi)=\widehat{\mu}(\xi)\widehat{\nu}(\xi)$ for all $\xi\in\R^d$.

On the existence of general infinite convolutions, we need the following theorem, in which statement (1) is a consequence of Kolmogorov's three series theorem (see for examples \cite[Theorem 34]{JW35} and \cite[Theorem 3.1]{LMW22}), and statements (2) and (3) follow from a similar proof of \cite[Theorem 1.1]{LMW22}. Here we use $B(r)$ to denote the closed ball centered at the original point $\mathbf{0}\in\R^d$ with radius $r$.

\begin{theorem}\label{exist} Let $d\in\N$ and $\mu_1,\mu_2,\cdots\in\cP(\R^d)$.
\begin{itemize}
\item[\emph{(1)}] Fix a constance $r>0$. The infinite convolution $\mu_1*\mu_2*\cdots$ exists if and only if the following three series all converge:
$$\textcircled{\footnotesize{$1$}}\text{ }\sum_{k=1}^\infty\mu_k\big(\R^d\setminus B(r)\big);\quad\quad\textcircled{\footnotesize{$2$}}\text{ }\sum_{k=1}^\infty\int_{B(r)}x\text{ }\mathrm{d}\mu_k(x);$$
$$\textcircled{\footnotesize{$3$}}\text{ }\sum_{k=1}^\infty\Big(\int_{B(r)}|x|^2\text{ }\mathrm{d}\mu_k(x)-\Big|\int_{B(r)}x\text{ }\mathrm{d}\mu_k(x)\Big|^2\Big).$$
\item[\emph{(2)}] If
$$\sum_{k=1}^\infty\int_{\R^d}\frac{|x|}{1+|x|}\text{ }\mathrm{d}\mu_k(x)<\infty,$$
then $\mu_1*\mu_2*\cdots$ exists.
\item[\emph{(3)}] Suppose $\spt\mu_k\subseteq[0,\infty)^d$ for all $k\in\N$. Then $\mu_1*\mu_2*\cdots$ exists if and only if
$$\sum_{k=1}^\infty\int_{\R^d}\frac{|x|}{1+|x|}\text{ }\mathrm{d}\mu_k(x)<\infty.$$
\end{itemize}
\end{theorem}

We emphasize that in Theorem \ref{exist} (2), without the condition $\spt\mu_k\subseteq[0,\infty)^d$ for all $k\in\N$, the convergence of $\sum_{k=1}^\infty\int_{\R^d}\frac{|x|}{1+|x|}\text{ }\mathrm{d}\mu_k(x)$ is enough to guarantee the existence of the infinite convolution $\mu_1*\mu_2*\cdots$. Immediately we get the following, which will be used in the proof of Theorem \ref{spectrality}.

\begin{corollary}\label{exist-cor} Let $d\in\N$ and $A_1,A_2,\cdots\subseteq\R^d$ be non-empty finite sets. If
$$\sum_{k=1}^\infty\frac{1}{\#A_k}\sum_{a\in A_k}\frac{|a|}{1+|a|}<\infty,$$
then $\delta_{A_1}*\delta_{A_2}*\cdots$ exists.
\end{corollary}

The equi-positivity property was introduced in \cite{AFL19,DHL19} and used to study the spectrality of fractal measures with compact supports. Then it was generalized to the following version in \cite{LMW24,LW24} which can also be used to study the spectrality of measures without compact supports.

\begin{definition}[Equi-positive]
A family $\Phi\subseteq\cP(\R^d)$ is called \textit{equi-positive} if there exist $\epsilon>0$ and $\delta>0$ such that for each $x\in[0,1)^d$ and $\mu\in\Phi$, there exists an integer vector $k_{x,\mu}\in\Z^d$ such that
$$|\widehat{\mu}(x+y+k_{x,\mu})|\ge\epsilon$$
for all $y\in\R^d$ with $|y|<\delta$, where $k_{x,\mu}=\mathbf{0}$ for $x=\mathbf{0}$.
\end{definition}

Given a sequence $\{R_k\}_{k\ge1}$ of $d\times d$ invertible real matrices and a sequence $\{B_k\}_{k\ge1}$ of non-empty finite subsets of $\R^d$, suppose that the infinite convolution
\begin{equation}\label{mu}
\mu:=\delta_{R_1^{-1}B_1}*\delta_{R_1^{-1}R_2^{-1}B_2}*\delta_{R_1^{-1}R_2^{-1}R_3^{-1}B_3}*\cdots
\end{equation}
exists. For each $n\in\N$, write the \textit{tail} of $\mu$ by
$$\mu_{>n}:=\delta_{R_1^{-1}R_2^{-1}\cdots R_{n+1}^{-1}B_{n+1}}*\delta_{R_1^{-1}R_2^{-1}\cdots R_{n+2}^{-1}B_{n+2}}*\delta_{R_1^{-1}R_2^{-1}\cdots R_{n+3}^{-1}B_{n+3}}*\cdots$$
and define the \textit{push-forward measure} of $\mu_{>n}$ by
\begin{equation}\label{nu >n}
\nu_{>n}(\text{ }\mathbf{\cdot}\text{ }):=\mu_{>n}(R_1^{-1}R_2^{-1}\cdots R_n^{-1}\text{ }\mathbf{\cdot}\text{ }),
\end{equation}
that is,
$$\nu_{>n}=\delta_{R_{n+1}^{-1}B_{n+1}}*\delta_{R_{n+1}^{-1}R_{n+2}^{-1}B_{n+2}}*\delta_{R_{n+1}^{-1}R_{n+2}^{-1}R_{n+3}^{-1}B_{n+3}}*\cdots.$$
In the proof of \cite[Theorem 1.1]{LW24}, Li and Wang actually showed the following. (See \cite[Theorem 1.4]{LMW24} and \cite[Theorem 4.2]{LMW22} for the version in $\R$.)

\begin{theorem}[\cite{LW24} ]\label{cite}
Let $d\in\N$ and $\{(R_k,B_k)\}_{k\ge1}$ be a sequence of admissible pairs in $\R^d$. Suppose that the infinite convolution $\mu$ defined in (\ref{mu}) exists, and
$$\lim_{n\to\infty}|(R_n^T)^{-1}\cdots(R_1^T)^{-1}x|=0\quad\text{for all }x\in\R^d.$$
Let $\{\nu_{>n}\}_{n\ge1}$ be defined in (\ref{nu >n}). If there exists a subsequence $\{\nu_{>n_j}\}_{j\ge1}$ which is equi-positive, then $\mu$ is a spectral measure with a spectrum in $\Z^d$.
\end{theorem}

For $B_1,B_2,\cdots\subseteq\R^d$, define
$$\lim_{n\to\infty}B_n:=\Big\{x\in\R^d:\exists b_n\in B_n\text{ for each }n\in\N\text{ s.t. }x=\lim_{n\to\infty}b_n\Big\}.$$

The following old result \cite[Theorem 3]{JW35} given by Jessen and Wintner in 1935 will be used in the proof of Theorem \ref{spt}.

\begin{theorem}[\cite{JW35} ]\label{old} Let $d\in\N$ and $\mu_1,\mu_2,\cdots\in\cP(\R^d)$ such that $\mu_1*\mu_2*\cdots$ exists. Then
$$\spt(\mu_1*\mu_2*\cdots)=\lim_{n\to\infty}(\spt\mu_1+\cdots+\spt\mu_n).$$
\end{theorem}

The following is the well-known Stolz-Ces\`aro Theorem.

\begin{theorem}\label{SC} Let $\beta_1,\beta_2,\beta_3,\cdots\in(0,\infty)$ such that $\sum_{n=1}^\infty\beta_n=\infty$ and let $\alpha_1,\alpha_2,\alpha_3,\cdots\in\R$. Then
$$\varliminf_{n\to\infty}\frac{\alpha_1+\alpha_2+\cdots+\alpha_n}{\beta_1+\beta_2+\cdots+\beta_n}\ge\varliminf_{n\to\infty}\frac{\alpha_n}{\beta_n}\quad\text{and}\quad\varlimsup_{n\to\infty}\frac{\alpha_1+\alpha_2+\cdots+\alpha_n}{\beta_1+\beta_2+\cdots+\beta_n}\le\varlimsup_{n\to\infty}\frac{\alpha_n}{\beta_n}.$$
In particular, if $\lim_{n\to\infty}\frac{\alpha_n}{\beta_n}$ exists, then
$$\lim_{n\to\infty}\frac{\alpha_1+\alpha_2+\cdots+\alpha_n}{\beta_1+\beta_2+\cdots+\beta_n}=\lim_{n\to\infty}\frac{\alpha_n}{\beta_n}.$$
\end{theorem}

We present two useful facts in the following to end this section.

\begin{proposition}[Lagrange's trigonometric equality]\label{Lagrange} For all $\theta\in\R\setminus\{2k\pi:k\in\Z\}$ and $n\in\{0,1,2,\cdots\}$, we have
$$\sum_{k=0}^n\sin k\theta=\frac{\cos\frac{1}{2}\theta-\cos((n+\frac{1}{2})\theta)}{2\sin\frac{1}{2}\theta}\quad\text{and}\quad\sum_{k=0}^n\cos k\theta=\frac{\sin\frac{1}{2}\theta+\sin((n+\frac{1}{2})\theta)}{2\sin\frac{1}{2}\theta}.$$
\end{proposition}

\begin{proposition}\label{diff-quot} For all $n\in\N$, let $a_n>c_n>0$ and $b_n>d_n>0$ with $\lim_{n\to\infty}\frac{a_n}{b_n}=\lim_{n\to\infty}\frac{c_n}{d_n}=r\in[0,\infty)$. If $\varliminf_{n\to\infty}\frac{b_n}{d_n}>1$, then $\lim_{n\to\infty}\frac{a_n-c_n}{b_n-d_n}=r$.
\end{proposition}
\begin{proof} By $\varliminf_{n\to\infty}\frac{b_n}{d_n}>1$, there exist $N_0,k\in\N$ such that for all $n>N_0$ we have $\frac{b_n}{d_n}>1+\frac{2}{k-1}$, which is equivalent to
\begin{equation}\label{d<b}
(k+1)d_n<(k-1)b_n.
\end{equation}
Arbitrarily take $\epsilon>0$. By $\lim_{n\to\infty}\frac{a_n}{b_n}=\lim_{n\to\infty}\frac{c_n}{d_n}=r$, there exists $N>N_0$ such that for all $n>N$ we have
$$\left\{\begin{array}{l}
r-\frac{\epsilon}{k}<\frac{a_n}{b_n}<r+\frac{\epsilon}{k},\\
r-\frac{\epsilon}{k}<\frac{c_n}{d_n}<r+\frac{\epsilon}{k},
\end{array}\right.\quad\text{i.e.,}\quad
\left\{\begin{array}{rrrrr}
(r-\frac{\epsilon}{k})b_n&<&a_n&<&(r+\frac{\epsilon}{k})b_n,\\
-(r+\frac{\epsilon}{k})d_n&<&-c_n&<&-(r-\frac{\epsilon}{k})d_n,
\end{array}\right.$$
which imply
$$(r-\epsilon)(b_n-d_n)\overset{\text{by (\ref{d<b})}}{<}(r-\frac{\epsilon}{k})b_n-(r+\frac{\epsilon}{k})d_n<a_n-c_n<(r+\frac{\epsilon}{k})b_n-(r-\frac{\epsilon}{k})d_n\overset{\text{by (\ref{d<b})}}{<}(r+\epsilon)(b_n-d_n).$$
We get $r-\epsilon<\frac{a_n-c_n}{b_n-d_n}<r+\epsilon$ for all $n>N$. Therefore $\lim_{n\to\infty}\frac{a_n-c_n}{b_n-d_n}=r$.
\end{proof}

In Proposition \ref{diff-quot}, in order to get the conclusion $\lim_{n\to\infty}\frac{a_n-c_n}{b_n-d_n}=r$, the condition $\varliminf_{n\to\infty}\frac{b_n}{d_n}>1$ can not be omitted. Otherwise, we can take $a_n=10^n+n$, $b_n=10^n+1$ and $c_n=d_n=10^n$ for all $n\in\N$. Then $a_n>c_n>0$ and $b_n>d_n>0$ for all $n\in\N$ with $\lim_{n\to\infty}\frac{a_n}{b_n}=\lim_{n\to\infty}\frac{c_n}{d_n}=1$. But $\lim_{n\to\infty}\frac{a_n-c_n}{b_n-d_n}=\infty$.

\section{Proof of Theorem \ref{spectrality}}
\indent

First we give the following lemma. In the proof we will see the intricacy of the high-dimensional case, especially in the estimation of the lower bound of $|\widehat{\nu}_{>n}(\xi)|$ for $\xi\in[-2/3,2/3]^d$.

\begin{lemma}\label{equi-lemma}
Let $d\in\N$, $\{m_k\}_{k\ge1}$ be a sequence of positive integers no less than $2$, $\{R_k\}_{k\ge1}$ be a sequence of $d\times d$ invertible real matrices, $\{B_k\}_{k\ge1}$ be a sequence of nearly $d$-th power lattices with respect to $\{m_k\}_{k\ge1}$ and $\{R_k\}_{k\ge1}$, and $\{\nu_{>n}\}_{n\ge1}$ be given by (\ref{nu >n}). If $[-m_k,m_k]^d\subseteq R_k^T[-1,1]^d$ for every $k\in\N$, then there exists $n_0\in\N$ such that $\{\nu_{>n}\}_{n\ge n_0}$ is equi-positive.
\end{lemma}
\begin{proof}Let $c_k:=\#(B_k\setminus\{0,1,\cdots,m_k-1\}^d)$ for all $k\in\N$.
\vspace{5pt}
\newline(1) Prove that for all $k\in\N$ and $\xi=(\xi_1,\xi_2,\cdots,\xi_d)\in[-\frac{\sqrt{6}}{m_k\pi},\frac{\sqrt{6}}{m_k\pi}]^d$ we have
$$|\widehat{\delta}_{B_k}(\xi)|\ge\prod_{j=1}^d\Big(1-\frac{m_k^2\pi^2\xi_j^2}{6}\Big)-\frac{2c_k}{m_k^d}.$$
Let $k\in\N$ and $\xi\in[-\frac{\sqrt{6}}{m_k\pi},\frac{\sqrt{6}}{m_k\pi}]^d$. Then
$$\begin{aligned}
|\widehat{\delta}_{B_k}(\xi)|&=\Big|\int_{\R^d}e^{-2\pi i<\xi,x>}d\delta_{B_k}(x)\Big|=\Big|\frac{1}{\#B_k}\sum_{b\in B_k}e^{-2\pi i<b,\xi>}\Big|\\
&\ge\frac{1}{m_k^d}\Big|\sum_{b\in\{0,1,\cdots,m_k-1\}^d}e^{-2\pi i<b,\xi>}\Big|-\frac{1}{m_k^d}\Big|\sum_{b\in\{0,1,\cdots,m_k-1\}^d}e^{-2\pi i<b,\xi>}-\sum_{b\in B_k}e^{-2\pi i<b,\xi>}\Big|\\
&\ge\frac{1}{m_k^d}\Big|\sum_{b_1,b_2,\cdots,b_d\in\{0,1,\cdots,m_k-1\}}e^{-2\pi i(b_1\xi_1+b_2\xi_2+\cdots b_d\xi_d)}\Big|-\frac{2}{m_k^d}\cdot\#\big(B_k\setminus\{0,1,\cdots,m_k-1\}^d\big)\\
&=\Big|\frac{1}{m_k}\sum_{b_1=0}^{m_k-1}e^{-2\pi ib_1\xi_1}\Big|\cdot\Big|\frac{1}{m_k}\sum_{b_2=0}^{m_k-1}e^{-2\pi ib_2\xi_2}\Big|\cdots\Big|\frac{1}{m_k}\sum_{b_d=0}^{m_k-1}e^{-2\pi ib_d\xi_d}\Big|-\frac{2c_k}{m_k^d}\\
&\ge\Big(1-\frac{m_k^2\pi^2\xi_1^2}{6}\Big)\Big(1-\frac{m_k^2\pi^2\xi_2^2}{6}\Big)\cdots\Big(1-\frac{m_k^2\pi^2\xi_d^2}{6}\Big)-\frac{2c_k}{m_k^d}
\end{aligned}$$
where the last inequality follows from the fact that for all $j\in\{1,2,\cdots,d\}$ we can prove
$$\Big|\frac{1}{m_k}\sum_{b=0}^{m_k-1}e^{-2\pi ib\xi_j}\Big|\ge1-\frac{m_k^2\pi^2\xi_j^2}{6}\ge0.$$
Since the second inequality follows immediately from $\xi\in[-\frac{\sqrt{6}}{m_k\pi},\frac{\sqrt{6}}{m_k\pi}]^d$, it suffices to prove the first one. If $\xi_j\in\Z$, the first inequality obviously holds. If $\xi_j\notin\Z$, by Lagrange's trigonometric equality in Proposition \ref{Lagrange} we get
$$\frac{1}{m_k}\Big|\sum_{b=0}^{m_k-1}e^{-2\pi ib\xi_j}\Big|=\frac{1}{m_k}\Big|\frac{\sin(m_k\pi\xi_j)}{\sin(\pi\xi_j)}\Big|\overset{(\star)}{\ge}\Big|\frac{\sin(m_k\pi\xi_j)}{m_k\pi\xi_j}\Big|\overset{(\star\star)}{\ge}1-\frac{m_k^2\pi^2\xi_j^2}{6},$$
where ($\star$) and ($\star\star$) follow respectively from $|\sin x|\le|x|$ and $|\frac{\sin x}{x}|\ge1-\frac{x^2}{6}$ for all $x\in\R\setminus\{0\}$.
\vspace{5pt}
\newline(2) Prove that for all $n,k\in\N$ we have
$$(\mathbf{R}_{n,n+k}^T)^{-1}\big[-\frac{2}{3},\frac{2}{3}\big]^d\subseteq\big[-\frac{2}{3m_{n+1}\cdots m_{n+k}},\frac{2}{3m_{n+1}\cdots m_{n+k}}\big]^d$$
where $\mathbf{R}_{n,n+k}:=R_{n+k}R_{n+k-1}\cdots R_{n+1}$. By the linearity of $(\mathbf{R}_{n,n+k}^T)^{-1}$, it suffices to prove
\begin{equation}\label{nk subset}
(R_{n+k}^T)^{-1}\cdots(R_{n+1}^T)^{-1}[-1,1]^d\subseteq\big[-\frac{1}{m_{n+1}\cdots m_{n+k}},\frac{1}{m_{n+1}\cdots m_{n+k}}\big]^d.
\end{equation}
Note that for all $k\in\N$ we have the condition $[-m_k,m_k]^d\subseteq R_k^T[-1,1]^d$, which is equivalent to
\begin{equation}\label{k subset}
(R_k^T)^{-1}[-1,1]^d\subseteq\big[-\frac{1}{m_k},\frac{1}{m_k}\big]^d.
\end{equation}
Thus
$$(R_{n+1}^T)^{-1}[-1,1]^d\subseteq\big[-\frac{1}{m_{n+1}},\frac{1}{m_{n+1}}\big]^d,$$
and then
$$(R_{n+2}^T)^{-1}(R_{n+1}^T)^{-1}[-1,1]^d\subseteq(R_{n+2}^T)^{-1}\big[-\frac{1}{m_{n+1}},\frac{1}{m_{n+1}}\big]^d\subseteq\big[-\frac{1}{m_{n+1}m_{n+2}},\frac{1}{m_{n+1}m_{n+2}}\big]^d,$$
where the last inclusion follows from (\ref{k subset}) and the linearity of $(R_{n+2}^T)^{-1}$. Repeating this process for finitely many times, we get (\ref{nk subset}).
\vspace{5pt}
\newline(3) Prove that there exists $n_0\in\N$ and $\epsilon>0$ such that for all $n\ge n_0$ and $\xi\in[-\frac{2}{3},\frac{2}{3}]^d$ we have $|\widehat{\nu}_{>n}(\xi)|\ge\epsilon$.

In fact, by $\sum_{k=1}^\infty\frac{c_k}{m_k^d}<\infty$ there exists $n_0\in\N$ such that for all $k\ge n_0$ we have
$$\frac{2c_k}{m_k^d}<\frac{2\pi^2}{27}\Big(1-\frac{2\pi^2}{27}\Big)^d.$$
Let $n\ge n_0$ and $\xi\in[-\frac{2}{3},\frac{2}{3}]^d$. By
$$\nu_{>n}=\delta_{\mathbf{R}_{n,n+1}^{-1}B_{n+1}}*\delta_{\mathbf{R}_{n,n+2}^{-1}B_{n+2}}*\delta_{\mathbf{R}_{n,n+3}^{-1}B_{n+3}}*\cdots$$
we get
\begin{equation}\label{prod}
|\widehat{\nu}_{>n}(\xi)|=\Big|\prod_{k=1}^\infty\widehat{\delta}_{\mathbf{R}_{n,n+k}^{-1}B_{n+k}}(\xi)\Big|=\Big|\prod_{k=1}^\infty\widehat{\delta}_{B_{n+k}}((\mathbf{R}_{n,n+k}^{-1})^T\xi)\Big|=\prod_{k=1}^\infty\Big|\widehat{\delta}_{B_{n+k}}((\mathbf{R}_{n,n+k}^T)^{-1}\xi)\Big|.
\end{equation}
For all $k\in\N$, it follows from $\xi\in[-\frac{2}{3},\frac{2}{3}]^d$, $m_{n+1},\cdots,m_{n+k-1}\ge2$ and (2) that
$$(\mathbf{R}_{n,n+k}^T)^{-1}\xi\in\Big[-\frac{2}{3\cdot2^{k-1}m_{n+k}},\frac{2}{3\cdot2^{k-1}m_{n+k}}\Big]^d\subseteq\Big[-\frac{2}{3m_{n+k}},\frac{2}{3m_{n+k}}\Big]^d\subseteq\Big[-\frac{\sqrt{6}}{m_{n+k}\pi},\frac{\sqrt{6}}{m_{n+k}\pi}\Big]^d.$$
Use $\big((\mathbf{R}_{n,n+k}^T)^{-1}\xi\big)_j$ to denote the $j$th coordinate of $(\mathbf{R}_{n,n+k}^T)^{-1}\xi$ for $j\in\{1,\cdots,d\}$. For all $k\in\N$, by (1) we get
\begin{equation}\label{ge}
\begin{aligned}
\Big|\widehat{\delta}_{B_{n+k}}\big((\mathbf{R}_{n,n+k}^T)^{-1}\xi\big)\Big|&\ge\prod_{j=1}^d\Big(1-\frac{m_{n+k}^2\pi^2\big((\mathbf{R}_{n,n+k}^T)^{-1}\xi\big)_j^2}{6}\Big)-\frac{2c_{n+k}}{m_{n+k}^d}\\
&\ge\prod_{j=1}^d\Big(1-\frac{m_{n+k}^2\pi^2}{6}\cdot\big(\frac{2}{3\cdot2^{k-1}m_{n+k}}\big)^2\Big)-\frac{2c_{n+k}}{m_{n+k}^d}\\
&=\Big(1-\frac{2\pi^2}{27\cdot4^{k-1}}\Big)^d-\frac{2c_{n+k}}{m_{n+k}^d}\\
&>\Big(1-\frac{2\pi^2}{27}\Big)^d-\frac{2\pi^2}{27}\Big(1-\frac{2\pi^2}{27}\Big)^d\\
&=\Big(1-\frac{2\pi^2}{27}\Big)^{d+1}\\
&>0.
\end{aligned}
\end{equation}
Let
$$\alpha:=\frac{(d+1)\ln(1-\frac{2\pi^2}{27})}{(1-\frac{2\pi^2}{27})^{d+1}-1}>0.$$
Then one can verify
\begin{equation}\label{x ge}
x\ge e^{\alpha(x-1)}>0\quad\text{for all }x\in\Big[\big(1-\frac{2\pi^2}{27}\big)^{d+1},1\Big].
\end{equation}
It follows from (\ref{prod}), (\ref{ge}) and (\ref{x ge}) that
$$\begin{aligned}
|\widehat{\nu}_{>n}(\xi)|&\ge\prod_{k=1}^\infty\Big(\big(1-\frac{2\pi^2}{27\cdot4^{k-1}}\big)^d-\frac{2c_{n+k}}{m_{n+k}^d}\Big)\\
&\ge\prod_{k=1}^\infty\exp\Big(\alpha\big((1-\frac{2\pi^2}{27\cdot4^{k-1}})^d-1-\frac{2c_{n+k}}{m_{n+k}^d}\big)\Big)\\
&\ge\prod_{k=1}^\infty\exp\Big(\alpha\big(1-(1+\frac{2\pi^2}{27\cdot4^{k-1}})^d-\frac{2c_{n+k}}{m_{n+k}^d}\big)\Big)\\
&=\prod_{k=1}^\infty\exp\Big(-\alpha\big(\sum_{j=1}^d\binom{d}{j}(\frac{2\pi^2}{27\cdot4^{k-1}})^j+\frac{2c_{n+k}}{m_{n+k}^d}\big)\Big)\\
&=\exp\Big(-\alpha\sum_{k=1}^\infty\big(\sum_{j=1}^d\binom{d}{j}(\frac{2\pi^2}{27})^j\cdot\frac{1}{4^{j(k-1)}}+\frac{2c_{n+k}}{m_{n+k}^d}\big)\Big)\\
&=\exp\Big(-\alpha\big(\sum_{j=1}^d\binom{d}{j}(\frac{2\pi^2}{27})^j\sum_{k=1}^\infty(\frac{1}{4^j})^{k-1}+2\sum_{k=1}^\infty\frac{c_{n+k}}{m_{n+k}^d}\big)\Big)\\
&\ge\exp\Big(-\alpha\sum_{j=1}^d\binom{d}{j}\frac{8^j\pi^{2j}}{27^j(4^j-1)}-2\alpha\sum_{k=1}^\infty\frac{c_k}{m_k^d}\Big)\xlongequal[]{\text{denoted by}}:\epsilon>0
\end{aligned}$$
for all $n\ge n_0$ and $\xi\in\big[-\frac{2}{3},\frac{2}{3}\big]^d$, where $\binom{d}{j}:=\frac{d!}{(d-j)!\cdot j!}$.
\newpage\noindent(4) Prove that $\{\nu_{>n}\}_{n\ge n_0}$ is equi-positive.
\newline Let $\delta=\frac{1}{6}$. For each $x=(x_1,\cdots,x_d)\in[0,1)^d$, define $k=k(x)=(k_1.\cdots,k_d)\in\Z^d$ by
$$k_j:=\left\{\begin{array}{ll}
0 & \mbox{if } x_j\in[0,\frac{1}{2})\\
-1 & \mbox{if } x_j\in[\frac{1}{2},1)
\end{array}\right.\quad\text{for all }j\in\{1,\cdots,d\}.$$
Then for all $x\in[0,1)^d$, $n\ge n_0$ and $y\in\R^d$ with $|y|<\delta$, we have $x+k(x)+y\in[-\frac{2}{3},\frac{2}{3}]^d$, and by the above (3) we get $|\widehat{\nu}_{>n}(x+k(x)+y)|\ge\epsilon$. Therefore $\{\nu_{>n}\}_{n\ge n_0}$ is equi-positive.
\end{proof}

Before deducing Theorem \ref{spectrality} from Lemma \ref{equi-lemma} and Theorem \ref{cite}, we need the following Propositions \ref{implies expanding}, \ref{mod HT} and \ref{R to RT} to deal with our high-dimensional case.

\begin{proposition}\label{implies expanding}
Let $d\in\N$, $R$ be a $d\times d$ real matrix and $C>0$. If $[-C,C]^d\subseteq R[-1,1]^d$, then all eigenvalues of $R$ have modulus no less than $C$.
\end{proposition}
\begin{proof} We use an iterative technique. Let $\lambda\in\C$ be an eigenvalue of $R$. Then there exists
\begin{equation}\label{z in}
z\in\big([-1,1]+i[-1,1]\big)^d\setminus\{\mathbf{0}\}
\end{equation}
such that $Rz=\lambda z$. We need to prove $|\lambda|\ge C$. Since $[-C,C]^d\subseteq R[-1,1]^d$ implies that $R$ is invertible, we get $CR^{-1}[-1,1]^d\subseteq[-1,1]^d$ and then
$$CR^{-1}\big([-1,1]+i[-1,1]\big)^d\subseteq\big([-1,1]+i[-1,1]\big)^d.$$
It follows from (\ref{z in}) that
$$(CR^{-1})^nz\in\big([-1,1]+i[-1,1]\big)^d$$
for all $n\in\N$. Since $Rz=\lambda z$ implies $R^nz=\lambda^nz$, we get
$$C^nz=\lambda^n(CR^{-1})^nz\in\lambda^n\big([-1,1]+i[-1,1]\big)^d$$
and then
$$C^n|z|\le|\lambda|^n\sqrt{2d}\quad\text{for all }n\in\N.$$
By $|z|\neq0$, we must have $C\le|\lambda|$.
\end{proof}

\begin{proposition}\label{mod HT} Let $d\in\N$, $m\ge2$ be an integer, $R$ be a $d\times d$ expanding matrix with integer entries which are all multiples of $m$, and $B\subseteq\R^d$ such that $B\equiv\{0,1,\cdots,m-1\}^d$ (mod $R\Z^d$). Then $B\subseteq\Z^d$ and $(R,B)$ is an admissible pair in $\R^d$.
\end{proposition}
\begin{proof} By $B\equiv\{0,1,\cdots,m-1\}^d$ (mod $R\Z^d$), for every $u\in\{0,1,\cdots,m-1\}^d$, there exists $z_u\in\Z^d$ such that $B=\big\{u+Rz_u:u\in\{0,1,\cdots,m-1\}^d\big\}\subseteq\Z^d$. Let $L:=\frac{1}{m}R^T\{0,1,\cdots,m-1\}^d$. We need to prove that the matrix
$$\Big[\frac{1}{\sqrt{m^d}}e^{-2\pi i<R^{-1}b,l>}\Big]_{b\in B,l\in L}$$
is unitary, and equivalently,
$$\frac{1}{\sqrt{m^d}}\Big[e^{-2\pi i<R^{-1}(u+Rz_u),\frac{1}{m}R^Tv>}\Big]_{u,v\in\{0,1,\cdots,m-1\}^d}$$
is unitary. Note that
$$e^{-2\pi i<R^{-1}(u+Rz_u),\frac{1}{m}R^Tv>}=e^{-2\pi i<R^{-1}u,\frac{1}{m}R^Tv>}\cdot e^{-2\pi i<z_u,\frac{1}{m}R^Tv>}\overset{(\star)}{=}e^{-\frac{2\pi i}{m}<R^{-1}u,R^Tv>}=e^{-\frac{2\pi i}{m}<u,v>}$$
for all $u,v\in\{0,1,\cdots,m-1\}^d$, where $(\star)$ follows from $<z_u,\frac{1}{m}R^Tv>\in\Z$ since all entries of $R^T$
\newpage\noindent are multiples of $m$. We only need to prove that
$$\frac{1}{\sqrt{m^d}}\Big[e^{-\frac{2\pi i}{m}<u,v>}\Big]_{u,v\in\{0,1,\cdots,m-1\}^d}$$
is unitary. It suffices to prove
$$\sum_{v\in\{0,1,\cdots,m-1\}^d}e^{-\frac{2\pi i}{m}<u^{(1)}-u^{(2)},v>}=0$$
for all $u^{(1)},u^{(2)}\in\{0,1,\cdots,m-1\}^d$ with $u^{(1)}\neq u^{(2)}$, which is equivalent to
$$\sum_{v\in\{0,1,\cdots,m-1\}^d}e^{\frac{2\pi i}{m}<u,v>}=0$$
for all $u\in\{-(m-1),\cdots,-1,0,1,\cdots,m-1\}^d\setminus\{0^d\}$. Let $u=(u_1,\cdots,u_d)$ with $u_1,\cdots,u_d\in\{-(m-1),\cdots,-1,0,1,\cdots,m-1\}$ and $u_k\neq0$ for some $k\in\{1,\cdots,d\}$. Then
$$\begin{aligned}
\sum_{v\in\{0,1,\cdots,m-1\}^d}e^{\frac{2\pi i}{m}<u,v>}&=\sum_{v_1=0}^{m-1}\sum_{v_2=0}^{m-1}\cdots\sum_{v_d=0}^{m-1}e^{\frac{2\pi i}{m}(u_1v_1+u_2v_2+\cdots+u_dv_d)}\\
&=\Big(\sum_{s=0}^{m-1}e^{su_1\cdot\frac{2\pi i}{m}}\Big)\Big(\sum_{s=0}^{m-1}e^{su_2\cdot\frac{2\pi i}{m}}\Big)\cdots\Big(\sum_{s=0}^{m-1}e^{su_d\cdot\frac{2\pi i}{m}}\Big).
\end{aligned}$$
We only need to prove $\sum_{s=0}^{m-1}e^{su_k\cdot\frac{2\pi i}{m}}=0$. It suffices to prove $\sum_{s=0}^{m-1}e^{st\cdot\frac{2\pi i}{m}}=0$ for all $t\in\{1,2,\cdots,m-1\}$. Let $t\in\{1,2,\cdots,m-1\}$ and write $t=t'r$, $m=m'r$ with $r,t',m'\in\N$ such that $t'$ and $m'$ are coprime. Then
$$\begin{aligned}
\sum_{s=0}^{m-1}e^{st\cdot\frac{2\pi i}{m}}&=\sum_{s=0}^{m'r-1}e^{st'\cdot\frac{2\pi i}{m'}}=\sum_{n=0}^{r-1}\sum_{s=nm'}^{(n+1)m'-1}e^{st'\cdot\frac{2\pi i}{m'}}\\
&=\sum_{n=0}^{r-1}\sum_{s=0}^{m'-1}e^{(nm'+s)t'\cdot\frac{2\pi i}{m'}}=\sum_{n=0}^{r-1}\sum_{s=0}^{m'-1}e^{st'\cdot\frac{2\pi i}{m'}}=r\sum_{s=0}^{m'-1}e^{st'\cdot\frac{2\pi i}{m'}}.
\end{aligned}$$
We only need to prove $\sum_{s=0}^{m'-1}e^{st'\cdot\frac{2\pi i}{m'}}=0$. Since $t'$ and $m'$ are coprime, we get
$$\{0,t',2t',\cdots,(m'-1)t'\}\equiv\{0,1,2,\cdots,m'-1\}\quad\text{(mod $m'$)}.$$
Thus
$$\sum_{s=0}^{m'-1}e^{st'\cdot\frac{2\pi i}{m'}}=\sum_{s=0}^{m'-1}e^{s\cdot\frac{2\pi i}{m'}}\overset{(\star)}{=}0,$$
where in $(\star)$ we use $m'\ge2$ since $m>t$ implies $m'>t'\ge1$.
\end{proof}

\begin{proposition}\label{R to RT} Let $d\in\N$, $P$ be a $d\times d$ real matrix and $c\ge0$. If $P[-1,1]^d\subseteq[-c,c]^d$, then $P^T[-1,1]^d\subseteq[-cd,cd]^d$.
\end{proposition}
\begin{proof} Suppose $P[-1,1]^d\subseteq[-c,c]^d$ and write
$$P=\left(\begin{matrix}
p_{11} & \cdots & p_{1d} \\
\vdots &        & \vdots \\
p_{d1} & \cdots & p_{dd}
\end{matrix}\right).$$
For all $s,t\in\{1,2,\cdots,d\}$, by $P(1,\cdots,1)^T\in[-c,c]^d$ we get
$$-c\le p_{s1}+\cdots+p_{s(t-1)}+p_{st}+p_{s(t+1)}+\cdots+p_{sd}\le c,$$
and by $P(1,\cdots,1,\overset{t}{-1},1,\cdots,1)^T\in[-c,c]^d$ we get
$$-c\le p_{s1}+\cdots+p_{s(t-1)}-p_{st}+p_{s(t+1)}+\cdots+p_{sd}\le c.$$
Therefore
$$-c\le p_{st}\le c\quad\text{for all }s,t\in\{1,2,\cdots,d\}.$$
It follows that for all $\tau_1,\cdots,\tau_d\in\{1,-1\}$ we have
$$P^T\left(\begin{matrix}
\tau_1 \\
\vdots \\
\tau_d
\end{matrix}\right)\in[-dc,dc]^d.$$
By the linearity of $P^T$ we get $P^T[-1,1]^d\subseteq[-cd,cd]^d$.
\end{proof}

Now we prove Theorem \ref{spectrality} to end this section.

\begin{proof}[Proof of Theorem \ref{spectrality}] (1) Prove that $\mu=\delta_{R_1^{-1}B_1}*\delta_{R_1^{-1}R_2^{-1}B_2}*\delta_{R_1^{-1}R_2^{-1}R_3^{-1}B_3}*\cdots$ exists.

By Corollary \ref{exist-cor}, we only need to prove
$$\sum_{k=1}^\infty\frac{1}{\#B_k}\sum_{a\in R_1^{-1}\cdots R_k^{-1}B_k}\frac{|a|}{1+|a|}<\infty,\quad\text{i.e.},\quad\sum_{k=1}^\infty\frac{1}{m_k^d}\sum_{x\in B_k}\frac{|R_1^{-1}\cdots R_k^{-1}x|}{1+|R_1^{-1}\cdots R_k^{-1}x|}<\infty.$$
Divide $B_k$ into two parts
$$B_{k,1}:=B_k\cap\{0,1,\cdots,m_k-1\}^d\quad\text{and}\quad B_{k,2}:=B_k\setminus\{0,1,\cdots,m_k-1\}^d.$$
Since
$$\sum_{k=1}^\infty\frac{1}{m_k^d}\sum_{x\in B_{k,2}}\frac{|R_1^{-1}\cdots R_k^{-1}x|}{1+|R_1^{-1}\cdots R_k^{-1}x|}\le\sum_{k=1}^\infty\frac{\#B_{k,2}}{m_k^d}\overset{\text{by (\ref{nearly})}}{<}\infty,$$
it suffices to show
$$\sum_{k=1}^\infty\frac{1}{m_k^d}\sum_{x\in B_{k,1}}\frac{|R_1^{-1}\cdots R_k^{-1}x|}{1+|R_1^{-1}\cdots R_k^{-1}x|}<\infty$$
in the following. In fact we have
$$\begin{aligned}
\sum_{k=1}^\infty\frac{1}{m_k^d}\sum_{x\in B_{k,1}}\frac{|R_1^{-1}\cdots R_k^{-1}x|}{1+|R_1^{-1}\cdots R_k^{-1}x|}&\le\sum_{k=1}^\infty\frac{1}{m_k^d}\sum_{x\in\{0,\cdots,m_k-1\}^d}|R_1^{-1}\cdots R_k^{-1}x|\\
&\overset{(\star)}{\le}\sum_{k=1}^\infty\frac{1}{m_k^d}\cdot m_k^d\cdot\frac{d\sqrt{d}}{m_1\cdots m_{k-1}}\le\sum_{k=1}^\infty\frac{d\sqrt{d}}{2^{k-1}}=2d\sqrt{d}<\infty,
\end{aligned}$$
where ($\star$) follows from
$$|R_1^{-1}\cdots R_k^{-1}x|\le\frac{d\sqrt{d}}{m_1\cdots m_{k-1}}\quad\text{for all }x\in\{0,\cdots,m_k-1\}^d\text{ and }k\in\N,$$
which can be proved as follows.

In fact, we only need to prove
$$R_1^{-1}\cdots R_k^{-1}[-m_k,m_k]^d\subseteq\Big[-\frac{d}{m_1\cdots m_{k-1}},\frac{d}{m_1\cdots m_{k-1}}\Big]^d\quad\text{for all }k\in\N.$$
By the linearity of $R_1^{-1}\cdots R_k^{-1}$ and Proposition \ref{R to RT}, it suffices to show
$$(R_k^{-1})^T\cdots(R_1^{-1})^T[-1,1]^d\subseteq\Big[-\frac{1}{m_1\cdots m_k},\frac{1}{m_1\cdots m_k}\Big]^d\quad\text{for all }k\in\N.$$
This follows immediately from the linearity of $(R_1^{-1})^T,\cdots,(R_k^{-1})^T$ and the fact that the condition $[-m_k,m_k]^d\subseteq R_k^T[-1,1]^d$ implies $(R_k^{-1})^T[-1,1]^d\subseteq[-\frac{1}{m_k},\frac{1}{m_k}]^d$ for every $k\in\N$.
\vspace{5pt}
\newline(2) Prove that $\mu$ is a spectral measure with a spectrum in $\Z^d$.
\begin{itemize}
\item[\textcircled{\footnotesize{$1$}}] Let $\{\nu_{>n}\}_{n\ge1}$ be given by (\ref{nu >n}). It follows from Lemma \ref{equi-lemma} that there exists $n_0\in\N$ such that $\{\nu_{>n}\}_{n\ge n_0}$ is equi-positive.
\item[\textcircled{\footnotesize{$2$}}] For every $k\in\N$, prove that $(R_k,B_k)$ is an admissible pair in $\R^d$.

In fact, by $[-m_k,m_k]^d\subseteq R_k^T[-1,1]^d$ and $m_k\ge2$, it follows from Proposition \ref{implies expanding} that all eigenvalues of $R_k^T$ have modulus no less than $2$. Noting that $R_k$ and $R_k^T$ have the same eigenvalues, $R_k$ must be expanding. Since $B_k\equiv\{0,1,\cdots,m_k-1\}^d$ (mod $R_k\Z^d$), by Proposition \ref{mod HT} we know that $(R_k,B_k)$ is an admissible pair in $\R^d$.
\item[\textcircled{\footnotesize{$3$}}] We have $\lim_{n\to\infty}|(R_n^T)^{-1}\cdots(R_1^T)^{-1}x|=0$ for all $x\in\R^d$, since $[-m_k,m_k]^d\subseteq R_k^T[-1,1]^d$ for all $k\in\N$ imply $(R_n^T)^{-1}\cdots(R_1^T)^{-1}[-1,1]^d\subseteq\big[-\frac{1}{m_1\cdots m_n},\frac{1}{m_1\cdots m_n}\big]^d\subseteq\big[-\frac{1}{2^n},\frac{1}{2^n}\big]^d$ for all $n\in\N$.
\end{itemize}
Combining the above \textcircled{\footnotesize{$1$}}, \textcircled{\footnotesize{$2$}}, \textcircled{\footnotesize{$3$}} and (1), by Theorem \ref{cite} we know that $\mu$ is a spectral measure with a spectrum in $\Z^d$.
\end{proof}

\section{Proofs of Theorem \ref{spt} and Corollary \ref{spt-1}}
\indent

Recall that for $A_1,A_2,\cdots\subseteq\R^d$,
$$\sum_{k=1}^\infty A_k:=\Big\{x\in\R^d:\exists a_k\in A_k\text{ for each }k\in\N\text{ s.t. }x=\sum_{k=1}^\infty a_k\Big\},$$
and for $B_1,B_2,\cdots\subseteq\R^d$,
$$\lim_{n\to\infty}B_n:=\Big\{x\in\R^d:\exists b_n\in B_n\text{ for each }n\in\N\text{ s.t. }x=\lim_{n\to\infty}b_n\Big\}.$$

Before proving Theorem \ref{spt}, we give the following two propositions first.

\begin{proposition}\label{+=}
Let $A_1,A_2,\cdots\subseteq\R^d$.
\begin{itemize}
\item[\emph{(1)}] We have
$$\lim_{n\to\infty}(A_1+\cdots+A_n)\supseteq\overline{\sum_{k=1}^\infty A_k}.$$
\item[\emph{(2)}] If $\sum_{k=1}^\infty A_k\neq\emptyset$, then
$$\lim_{n\to\infty}(A_1+\cdots+A_n)=\overline{\sum_{k=1}^\infty A_k}.$$
\end{itemize}
\end{proposition}
\begin{proof} (1) Since $\sum_{k=1}^\infty A_k\subseteq\lim_{n\to\infty}(A_1+\cdots+A_n)$ is obvious, it suffices to prove that $\lim_{n\to\infty}(A_1+\cdots+A_n)$ is closed. In fact, for any $B_1,B_2,\cdots\subseteq\R^d$, we can prove that $\lim_{n\to\infty}B_n$ is a closed set. Let $x_1,x_2,\cdots\in\lim_{n\to\infty}B_n$ and $x\in\R^d$ such that $\lim_{k\to\infty}x_k=x$. It suffices to prove $x\in\lim_{n\to\infty}B_n$.

For each $k\in\N$, by $x_k\in\lim_{n\to\infty}B_n$, there exist $b_{k,1}\in B_1,b_{k,2}\in B_2,\cdots$ such that
$$\lim_{n\to\infty}b_{k,n}=x_k.$$

By $\lim_{n\to\infty}b_{1,n}=x_1$, there exists $N_1\ge1$ such that for all $n\ge N_1$ we have $|b_{1,n}-x_1|<1$.

By $\lim_{n\to\infty}b_{2,n}=x_2$, there exists $N_2>N_1$ such that for all $n\ge N_2$ we have $|b_{2,n}-x_2|<\frac{1}{2}$.

$\cdots$

Repeating this process, we can find $1\le N_1<N_2<N_3<\cdots$ such that
\begin{equation}\label{<1/k}
\text{for all }k\in\N\text{ and }n\ge N_k\text{ we have }|b_{k,n}-x_k|<\frac{1}{k}.
\end{equation}
For each $k\in\N$ and $n\in\{N_k,N_k+1,\cdots,N_{k+1}-1\}$, we define $c_n:=b_{k,n}\in B_n$. It suffices to prove $\lim_{n\to\infty}c_n=x$. Arbitrarily take $\epsilon>0$. By $\lim_{k\to\infty}x_k=x$, there exists $K>\frac{2}{\epsilon}$ such that
\begin{equation}\label{<e/2}
\text{for all }k\ge K\text{ we have }|x_k-x|<\frac{\epsilon}{2}.
\end{equation}
For each $n\ge N_K$, there exists $k\ge K$ such that $N_k\le n<N_{k+1}$ and then
\begin{equation}\label{<E/2}
|c_n-x_k|\overset{\text{by (\ref{<1/k})}}{<}\frac{1}{k}\le\frac{1}{K}<\frac{\epsilon}{2},
\end{equation}
which implies
$$|c_n-x|\le|c_n-x_k|+|x_k-x|<\frac{\epsilon}{2}+\frac{\epsilon}{2}=\epsilon,$$
where the second inequality follows from (\ref{<e/2}) and (\ref{<E/2}). Therefore $\lim_{n\to\infty}c_n=x$.
\vspace{5pt}
\newline(2) Suppose $\sum_{k=1}^\infty A_k\neq\emptyset$. By (1) we only need to prove $\lim_{n\to\infty}(A_1+\cdots+A_n)\subseteq\overline{\sum_{k=1}^\infty A_k}$. Let $x\in\lim_{n\to\infty}(A_1+\cdots+A_n)$. Then there exists $b_n\in A_1+\cdots+A_n$ for each $n\in\N$ such that $x=\lim_{n\to\infty}b_n$. By $\sum_{k=1}^\infty A_k\neq\emptyset$, there exist $a_1\in A_1$, $a_2\in A_2$, $\cdots$ such that $\sum_{k=1}^\infty a_k$ converges, which implies $\sum_{k=n+1}^\infty a_k\to\mathbf{0}$ as $n\to\infty$, where $\mathbf{0}$ denotes the zero vector in $\R^d$. For all $n\in\N$, define
$$x_n:=b_n+\sum_{k=n+1}^\infty a_k\in\sum_{k=1}^\infty A_k.$$
Then $\lim_{n\to\infty}x_n=x$ and we get $x\in\overline{\sum_{k=1}^\infty A_k}$.
\end{proof}

In Proposition \ref{+=} (2), the condition $\sum_{k=1}^\infty A_k\neq\emptyset$ can not be omitted. Otherwise, we can take $d=1$, $A_1=\{-1,0,1\}$ and $A_k=\{-1,1\}$ for all $k\ge2$. Then $\sum_{k=1}^\infty A_k=\emptyset$, and
$$\lim_{n\to\infty}(A_1+\cdots+A_n)=\lim_{n\to\infty}\{-n,\cdots,-1,0,1,\cdots,n\}=\Z\neq\overline{\sum_{k=1}^\infty A_k}.$$

Recall that: for all $A\subseteq\R^d$ and $j\in\{1,\cdots,d\}$, we denote $(A)_j=\big\{a_j\in\R:(a_1,\cdots,a_j,\cdots,a_d)\in A\big\}$; for a non-empty set $A\subseteq\R$, we use $\inf A\in\R\cup\{-\infty\}$ and $\sup A\in\R\cup\{+\infty\}$ respectively to denote the infimum and supremum of $A$, and use $\min A\in\R$ and $\max A\in\R$ respectively to denote the minimal and maximal of $A$ if they exist; for $x_1,x_2,\cdots\in\R\cup\{\pm\infty\}$, we say that $\sum_{k=1}^\infty x_k$ converges if $x_k\in\R$ (not $\pm\infty$) for all $k\in\N$ and the limit $\lim_{n\to\infty}\sum_{k=1}^nx_k$ exists (not $\pm\infty$).

The next proposition follows from some detailed convergence analysis, where the proof of statement (2) relies on a translation technique for sets of infinite sums.

\begin{proposition}\label{bounded-closed}
Let $A_1,A_2,\cdots\subseteq\R^d$ be non-empty sets.
\begin{itemize}
\item[\emph{(1)}] The set $\sum_{k=1}^\infty A_k$ is non-empty and bounded if and only if
\begin{equation}\label{both converge}
\sum_{k=1}^\infty\inf(A_k)_j\text{ and }\sum_{k=1}^\infty\sup(A_k)_j\text{ converge}\quad\text{for all }j\in\{1,\cdots,d\}.
\end{equation}
\item[\emph{(2)}] If $A_1,A_2,\cdots$ are all closed and (\ref{both converge}) holds, then $\sum_{k=1}^\infty A_k$ is non-empty and compact.
\end{itemize}
\end{proposition}
\begin{proof} (1) $\boxed{\Rightarrow}$ Suppose that $\sum_{k=1}^\infty A_k$ is non-empty and bounded. We need to prove (\ref{both converge}).
\newline\textcircled{\footnotesize{$1$}} First we prove that $A_1,A_2,\cdots$ are all bounded.
\newline In fact, by $\sum_{k=1}^\infty A_k\neq\emptyset$, there exist $a_1\in A_1$, $a_2\in A_2$, $\cdots$ such that $\sum_{k=1}^\infty a_k$ converges. For each $n\in\N$, since $\sum_{k=1}^{n-1}a_k+\sum_{k=n+1}^\infty a_k+A_n\subseteq\sum_{k=1}^\infty A_k$ and $\sum_{k=1}^\infty A_k$ is bounded, we know that $A_n$ is bounded.
\newline\textcircled{\footnotesize{$2$}} Now we prove (\ref{both converge}).
\newline Arbitrarily take $j\in\{1,\cdots,d\}$. In the following we only prove that $\sum_{k=1}^\infty\sup(A_k)_j$ converges since the ``inf'' case is similar. By the condition that $\sum_{k=1}^\infty A_k$ is non-empty and bounded, we can take $M_j:=\sup(\sum_{k=1}^\infty A_k)_j\in\R$. It suffices to prove $\lim_{n\to\infty}\sum_{k=1}^n\sup(A_k)_j=M_j$ in the following.
\begin{itemize}
\item[i)] Prove $\varlimsup_{n\to\infty}\sum_{k=1}^n\sup(A_k)_j\le M_j$ by contradiction.
\newline Assume $\varlimsup_{n\to\infty}\sum_{k=1}^n\sup(A_k)_j>M_j$. Then there exists $c_j>0$ and positive integers $n_1<n_2<n_3<\cdots$ such that
$$\sum_{k=1}^{n_p}\sup(A_k)_j>M_j+2c_j$$
for all $p\in\N$. By $\sum_{k=1}^\infty A_k\neq\emptyset$, there exists $a^{(k)}=(a^{(k)}_1,\cdots,a^{(k)}_d)\in A_k$ for each $k\in\N$ such that $\sum_{k=1}^\infty a^{(k)}$ converges, and then $\sum_{k=1}^\infty a^{(k)}_j$ also converges. This implies that there exists $q\in\N$ such that
$$\Big|\sum_{k=n_q+1}^\infty a^{(k)}_j\Big|<c_j.$$
For all $k\in\{1,2,\cdots,n_q\}$, it follows from $A_k\neq\emptyset$ and \textcircled{\footnotesize{$1$}} that we can take $x^{(k)}\in A_k$ such that
$$x^{(k)}_j>\sup(A_k)_j-\frac{c_j}{n_q}.$$
Let
$$s=(s_1,\cdots,s_d):=\sum_{k=1}^{n_q}x^{(k)}+\sum_{k=n_q+1}^\infty a^{(k)}\in\sum_{k=1}^\infty A_k.$$
Then
$$s_j=\sum_{k=1}^{n_q}x^{(k)}_j+\sum_{k=n_q+1}^\infty a^{(k)}_j>\sum_{k=1}^{n_q}\sup(A_k)_j-n_q\cdot\frac{c_j}{n_q}-c_j>M_j,$$
which contradicts $M_j=\sup(\sum_{k=1}^\infty A_k)_j$.
\item[ii)] Prove $\varliminf_{n\to\infty}\sum_{k=1}^n\sup(A_k)_j\ge M_j$.
\newline Arbitrarily take $s=(s_1,\cdots,s_d)\in\sum_{k=1}^\infty A_k$. We only need to prove $s_j\le\varliminf_{n\to\infty}\sum_{k=1}^n\sup(A_k)_j$ in the following. By $s\in\sum_{k=1}^\infty A_k$, there exists $x^{(k)}=(x^{(k)}_1,\cdots,x^{(k)}_d)\in A_k$ for each $k\in\N$ such that $s=\sum_{k=1}^\infty x^{(k)}$ and $s_j=\sum_{k=1}^\infty x^{(k)}_j$ converge. We get $s_j=\varliminf_{n\to\infty}\sum_{k=1}^n x^{(k)}_j\le\varliminf_{n\to\infty}\sum_{k=1}^n\sup(A_k)_j$.
\end{itemize}
$\boxed{\Leftarrow}$ Suppose that (\ref{both converge}) holds. We need to prove that $\sum_{k=1}^\infty A_k$ is non-empty and bounded.
\newline\textcircled{\footnotesize{$1$}} Prove $\sum_{k=1}^\infty A_k\neq\emptyset$.
\newline For each $k\in\N$, let $a^{(k)}=(a^{(k)}_1,\cdots,a^{(k)}_d)\in A_k$. We only need to prove that $\sum_{k=1}^\infty a^{(k)}$ converges. Arbitrarily take $j\in\{1,\cdots,d\}$. It suffices to show that $\sum_{k=1}^\infty a^{(k)}_j$ converges. Let $\epsilon>0$. By (\ref{both converge}) and Cauchy convergence criterion, there exists $N\in\N$ such that for all $m>n>N$ we have $|\sum_{k=n+1}^m\inf(A_k)_j|<\epsilon$, $|\sum_{k=n+1}^m\sup(A_k)_j|<\epsilon$, and then
$$-\epsilon<\sum_{k=n+1}^m\inf(A_k)_j\le\sum_{k=n+1}^ma^{(k)}_j\le\sum_{k=n+1}^m\sup(A_k)_j<\epsilon,$$
which implies $|\sum_{k=n+1}^ma^{(k)}_j|<\epsilon$. It follows again from Cauchy convergence criterion that $\sum_{k=1}^\infty a^{(k)}_j$ converges.
\newline\textcircled{\footnotesize{$2$}} Prove that $\sum_{k=1}^\infty A_k$ is bounded.
\newline Since both $m_j:=\sum_{k=1}^\infty\inf(A_k)_j$ and $M_j:=\sum_{k=1}^\infty\sup(A_k)_j$ converge for all $j\in\{1,\cdots,d\}$, one can easily verify $\sum_{k=1}^\infty A_k\subseteq[m_1,M_1]\times[m_2,M_2]\times\cdots\times[m_d,M_d]$. It follows that $\sum_{k=1}^\infty A_k$ is bounded.
\vspace{5pt}
\newline(2) We use a translation technique after proving a weaker conclusion.
\newline\textcircled{\footnotesize{$1$}} First we prove that: for any non-empty closed sets $B_1,B_2,\cdots\subseteq\R^d$ with $\sum_{k=1}^\infty\sup_{b\in B_k}|b|<\infty$, the set $\sum_{k=1}^\infty B_k$ is non-empty and compact.

In fact, by (1) and
$$\max\Big\{\sum_{k=1}^\infty\big|\inf(B_k)_j\big|,\sum_{k=1}^\infty\big|\sup(B_k)_j\big|\Big\}\le\sum_{k=1}^\infty\sup_{(b_1,\cdots,b_d)\in B_k}|b_j|\le\sum_{k=1}^\infty\sup_{b\in B_k}|b|<\infty$$
for all $j\in\{1,\cdots,d\}$, we know that $\sum_{k=1}^\infty B_k$ is non-empty and bounded. In the following we only need to prove that $\sum_{k=1}^\infty B_k$ is closed.

Let $x^{(1)},x^{(2)},\cdots\in\sum_{k=1}^\infty B_k$ such that $x^{(n)}$ converges to some $x\in\R^d$ as $n\to\infty$. We need to prove $x\in\sum_{k=1}^\infty B_k$. For each $n\in\N$, by $x^{(n)}\in\sum_{k=1}^\infty B_k$, there exist $x^{(n)}_1\in B_1$, $x^{(n)}_2\in B_2$, $\cdots$ such that $x^{(n)}=\sum_{k=1}^\infty x^{(n)}_k$ converges. Since $B_1,B_2,\cdots$ are closed sets and $\sum_{k=1}^\infty\sup_{b\in B_k}|b|<\infty$, we know that $B_1,B_2,\cdots$ are all compact.

By $x^{(1)}_1,x^{(2)}_1,x^{(3)}_1,\cdots\in B_1$, there exists a subsequence $x^{(p_1)}_1,x^{(p_2)}_1,x^{(p_3)}_1,\cdots$ converges to some $b_1\in B_1$.

By $x^{(p_1)}_2,x^{(p_2)}_2,x^{(p_3)}_2,\cdots\in B_2$, there exists a subsequence $x^{(p_{q_1})}_2,x^{(p_{q_2})}_2,x^{(p_{q_3})}_2,\cdots$ converges to some $b_2\in B_2$ with $q_1\ge2$.

By $x^{(p_{q_1})}_3$, $x^{(p_{q_2})}_3$, $x^{(p_{q_3})}_3$, $\cdots$ $\in B_3$, there exists a subsequence $x^{(p_{q_{r_1}})}_3,x^{(p_{q_{r_2}})}_3,x^{(p_{q_{r_3}})}_3,\cdots$ converges to some $b_3\in B_3$ with $r_1\ge2$.

$\cdots$

Repeat this process and take $n_1=p_1$, $n_2=p_{q_1}$, $n_3=p_{q_{r_1}}$, $\cdots$. We get $n_1<n_2<n_3<\cdots$ such that $x^{(n_1)}_1,x^{(n_2)}_1,x^{(n_3)}_1,\cdots\to b_1\in B_1$, $x^{(n_2)}_2,x^{(n_3)}_2,x^{(n_4)}_2,\cdots\to b_2\in B_2$, $x^{(n_3)}_3,x^{(n_4)}_3,x^{(n_5)}_3,\cdots\to b_3\in B_3$, $\cdots$. In the following we only need to prove $x=\sum_{k=1}^\infty b_k$, i.e., $\lim_{j\to\infty}\sum_{k=1}^j b_k=x$. Arbitrarily take $\epsilon>0$. By $\sum_{k=1}^\infty\max_{b\in B_k}|b|<\infty$, there exists $m\in\N$ such that $\sum_{k=m+1}^\infty\max_{b\in B_k}|b|<\epsilon$. Since $x^{(n_j)}\to x$ as $j\to\infty$, there exists an integer $J_0\ge m$ such that for all $j>J_0$ we have $|x^{(n_j)}-x|<\epsilon$. Besides, by $x^{(n_j)}_1\to b_1$, $x^{(n_j)}_2\to b_2$, $\cdots$, $x^{(n_j)}_m\to b_m$ as $j\to\infty$, there exists an integer $J\ge J_0$ such that for all $j>J$ we have
$$|x^{(n_j)}_1-b_1|<\frac{\epsilon}{m},\quad|x^{(n_j)}_2-b_2|<\frac{\epsilon}{m},\quad\cdots,\quad|x^{(n_j)}_m-b_m|<\frac{\epsilon}{m},$$
and then
\begin{align*}
\Big|\sum_{k=1}^j b_k-x\Big|&\le|x-x^{(n_j)}|+\Big|x^{(n_j)}-\sum_{k=1}^m x^{(n_j)}_k\Big|+\Big|\sum_{k=1}^m x^{(n_j)}_k-\sum_{k=1}^m b_k\Big|+\Big|\sum_{k=1}^m b_k-\sum_{k=1}^j b_k\Big|\\
&<\epsilon+\Big|\sum_{k=m+1}^\infty x^{(n_j)}_k\Big|+\Big|\sum_{k=1}^m(x^{(n_j)}_k-b_k)\Big|+\Big|\sum_{k=m+1}^j b_k\Big|\\
&\le\epsilon+\sum_{k=m+1}^\infty|x^{(n_j)}_k|+\sum_{k=1}^m|x^{(n_j)}_k-b_k|+\sum_{k=m+1}^j|b_k|\\
&<\epsilon+\sum_{k=m+1}^\infty\max_{b\in B_k}|b|+m\cdot\frac{\epsilon}{m}+\sum_{k=m+1}^\infty\max_{b\in B_k}|b|<4\epsilon.
\end{align*}
Therefore $\lim_{j\to\infty}\sum_{k=1}^j b_k=x$.
\newline\textcircled{\footnotesize{$2$}} Now we suppose that $A_1,A_2,\cdots\subseteq\R^d$ are all non-empty closed sets and (\ref{both converge}) holds. We need to prove that $\sum_{k=1}^\infty A_k$ is non-empty and compact. For each $k\in\N$, define a translation of $A_k$ by $B_k:=A_k-\big(\inf(A_k)_1,\cdots,\inf(A_k)_d\big)\subseteq\big[0,\sup(A_k)_1-\inf(A_k)_1\big]\times\cdots\times\big[0,\sup(A_k)_d-\inf(A_k)_d\big]$. Then $B_1,B_2,\cdots$ are all non-empty closed sets, and
$$\sum_{k=1}^\infty B_k=\sum_{k=1}^\infty A_k-\Big(\sum_{k=1}^\infty\inf(A_k)_1,\cdots,\sum_{k=1}^\infty\inf(A_k)_d\Big)$$
is a translation of $\sum_{k=1}^\infty A_k$. By \textcircled{\footnotesize{$1$}} and
$$\begin{aligned}
\sum_{k=1}^\infty\sup_{b\in B_k}|b|&\le\sum_{k=1}^\infty\sqrt{\big(\sup(A_k)_1-\inf(A_k)_1\big)^2+\cdots+\big(\sup(A_k)_d-\inf(A_k)_d\big)^2}\\
&\le\sum_{k=1}^\infty\Big(\big(\sup(A_k)_1-\inf(A_k)_1\big)+\cdots+\big(\sup(A_k)_d-\inf(A_k)_d\big)\Big)\overset{\text{by (\ref{both converge})}}{<}\infty,
\end{aligned}$$
we know that $\sum_{k=1}^\infty B_k$ is non-empty and compact, so is $\sum_{k=1}^\infty A_k$.
\end{proof}

Now we use Theorem \ref{old}, probability theory, Propositions \ref{+=} and \ref{bounded-closed} to prove Theorem \ref{spt}.

\begin{proof}[Proof of Theorem \ref{spt}] Let $d\in\N$ and $\mu_1,\mu_2,\cdots\in\cP(\R^d)$.
\vspace{5pt}
\newline(1) Suppose that $\mu_1*\mu_2*\cdots$ exists, and we need to prove $\spt(\mu_1*\mu_2*\cdots)=\overline{\sum_{k=1}^\infty\spt\mu_k}$.
\newline Since Theorem \ref{old} says that $\spt(\mu_1*\mu_2*\cdots)=\lim_{n\to\infty}(\spt\mu_1+\cdots+\spt\mu_n)$, by Proposition \ref{+=} (2), we only need to prove $\sum_{k=1}^\infty\spt\mu_k\neq\emptyset$ in the following. In fact, it follows from the argument above Theorem 3.1 in \cite{LMW22} that there exist a probability space $(\Omega,\cF,\P)$ and a sequence of independent random vectors $\{X_k\}_{k\ge1}$ such that for each $k\ge1$ the distribution of $X_k$ is $\mu_k=\P\circ X_k^{-1}$, and
\begin{equation}\label{a.s.}
\P\Big(\sum_{k=1}^\infty X_k\text{ converges}\Big)=1.
\end{equation}
By $\P(X_k\in\spt\mu_k)=\mu_k(\spt\mu_k)=1$ for every $k\in\N$ and (\ref{a.s.}), there exists $w\in\Omega$ such that $X_k(w)\in\spt\mu_k$ for all $k\in\N$ and $\sum_{k=1}^\infty X_k(w)$ converges. Therefore $\sum_{k=1}^\infty\spt\mu_k\neq\emptyset$.
\vspace{5pt}
\newline(2) Suppose that $\mu_1*\mu_2*\cdots$ exists with compact support. By (1) we get $\spt(\mu_1*\mu_2*\cdots)=\overline{\sum_{k=1}^\infty\spt\mu_k}$, and then $\sum_{k=1}^\infty\spt\mu_k$ is non-empty and bounded. It follows from Proposition \ref{bounded-closed} (1) that (\ref{spt inf-sup con}) holds. By Proposition \ref{bounded-closed} (2) we know that $\sum_{k=1}^\infty\spt\mu_k$ is closed, and then $\spt(\mu_1*\mu_2*\cdots)=\sum_{k=1}^\infty\spt\mu_k$.

To complete the proof, in the following we only need to suppose that (\ref{spt inf-sup con}) holds and prove that $\mu_1*\mu_2*\cdots$ exists with compact support. If we can prove that $\mu_1*\mu_2*\cdots$ exists, then by the above (1) and Proposition \ref{bounded-closed} (1) we know that $\spt(\mu_1*\mu_2*\cdots)$ is compact. Thus we only need to prove that $\mu_1*\mu_2*\cdots$ exists. For all $j\in\{1,\cdots,d\}$, by (\ref{spt inf-sup con}) we get $\lim_{k\to\infty}\inf(\spt\mu_k)_j=0$ and $\lim_{k\to\infty}\sup(\spt\mu_k)_j=0$, which imply that the sequences $\{\inf(\spt\mu_k)_j\}_{k\ge1}$ and $\{\sup(\spt\mu_k)_j\}_{k\ge1}$ are both bounded. Thus we can take
$$m_j:=\inf_{k\ge1}\inf(\spt\mu_k)_j\in\R\quad\text{and}\quad M_j:=\sup_{k\ge1}\sup(\spt\mu_k)_j\in\R$$
for all $j\in\{1,\cdots,d\}$. By
$$\spt\mu_k\subseteq[\inf(\spt\mu_k)_1,\sup(\spt\mu_k)_1]\times\cdots\times[\inf(\spt\mu_k)_d,\sup(\spt\mu_k)_d]\subseteq[m_1,M_1]\times\cdots\times[m_d,M_d]$$ for all $k\in\N$, we know that there exists $r>0$ such that
$$\bigcup_{k=1}^\infty\spt\mu_k\subseteq B(r).$$
To prove that $\mu_1*\mu_2*\cdots$ exists, by Theorem \ref{exist} (1), it suffices to consider the following \textcircled{\footnotesize{$1$}}, \textcircled{\footnotesize{$2$}} and \textcircled{\footnotesize{$3$}}.
\begin{itemize}
\item[\textcircled{\footnotesize{$1$}}] We have $\sum_{k=1}^\infty\mu_k\big(\R^d\setminus B(r)\big)=0$ since $\spt\mu_k\subseteq B(r)$ for all $k\in\N$.
\item[\textcircled{\footnotesize{$2$}}] Prove that $\sum_{k=1}^\infty\int x\text{ }\mathrm{d}\mu_k(x)$ converges.
\newline Arbitrarily take $j\in\{1,\cdots,d\}$. It suffices to show that $\sum_{k=1}^\infty\int x_j\text{ }\mathrm{d}\mu_k(x_1,\cdots,x_d)$ converges. Let $\epsilon>0$. By (\ref{spt inf-sup con}) and Cauchy convergence criterion, there exists $N\in\N$ such that for all $m>n>N$ we have $|\sum_{k=n+1}^m\inf(\spt\mu_k)_j|<\epsilon$ and $|\sum_{k=n+1}^m\sup(\spt\mu_k)_j|<\epsilon$. Note that for all $k\in\N$ and $x=(x_1,\cdots,x_d)\in\spt\mu_k$ we have $\inf(\spt\mu_k)_j\le x_j\le\sup(\spt\mu_k)_j$. Thus for all $m>n>N$ we have
$$-\epsilon<\sum_{k=n+1}^m\inf(\spt\mu_k)_j\le\sum_{k=n+1}^m\int x_j\text{ }\mathrm{d}\mu_k(x_1,\cdots,x_d)\le\sum_{k=n+1}^m\sup(\spt\mu_k)_j<\epsilon.$$
It follows again from Cauchy convergence criterion that $\sum_{k=1}^\infty\int x_j\text{ }\mathrm{d}\mu_k(x_1,\cdots,x_d)$ converges.
\item[\textcircled{\footnotesize{$3$}}] Prove that $\sum_{k=1}^\infty\Big(\int|x|^2\text{ }\mathrm{d}\mu_k(x)-\Big|\int x\text{ }\mathrm{d}\mu_k(x)\Big|^2\Big)$ converges. In fact,
$$\begin{aligned}
\sum_{k=1}^\infty&\Big(\int|x|^2\text{ }\mathrm{d}\mu_k(x)-\Big|\int x\text{ }\mathrm{d}\mu_k(x)\Big|^2\Big)=\sum_{k=1}^\infty\int\Big|x-\int y\text{ }\mathrm{d}\mu_k(y)\Big|^2\mathrm{d}\mu_k(x)\\
&=\sum_{k=1}^\infty\int\sum_{j=1}^d\Big(x_j-\int y_j\text{ }\mathrm{d}\mu_k(y_1,\cdots,y_d)\Big)^2\text{ }\mathrm{d}\mu_k(x_1,\cdots,x_d)\\
&=\sum_{k=1}^\infty\sum_{j=1}^d\int\Big(\int(x_j-y_j)\text{ }\mathrm{d}\mu_k(y_1,\cdots,y_d)\Big)^2\text{ }\mathrm{d}\mu_k(x_1,\cdots,x_d)\\
&\overset{(\star)}{\le}\sum_{k=1}^\infty\sum_{j=1}^d\int\Big(\int\big(\sup(\spt\mu_k)_j-\inf(\spt\mu_k)_j\big)\text{ }\mathrm{d}\mu_k(y)\Big)^2\text{ }\mathrm{d}\mu_k(x)\\
&=\sum_{k=1}^\infty\sum_{j=1}^d\big(\sup(\spt\mu_k)_j-\inf(\spt\mu_k)_j\big)^2\overset{(\star\star)}{<}\infty,
\end{aligned}$$
where ($\star$) follows from $\inf(\spt\mu_k)_j\le x_j,y_j\le\sup(\spt\mu_k)_j$ for all $x=(x_1,\cdots,x_d),y=(y_1,\cdots,y_d)\in\spt\mu_k$, $j\in\{1,\cdots,d\}$ and $k\in\N$, and ($\star\star$) follows from the fact that
\newpage
$$\sum_{k=1}^\infty\big(\sup(\spt\mu_k)_j-\inf(\spt\mu_k)_j\big)=\sum_{k=1}^\infty\sup(\spt\mu_k)_j-\sum_{k=1}^\infty\inf(\spt\mu_k)_j\overset{\text{by (\ref{spt inf-sup con})}}{<}\infty$$
implies
$$\sum_{k=1}^\infty\big(\sup(\spt\mu_k)_j-\inf(\spt\mu_k)_j\big)^2<\infty$$
for every $j\in\{1,\cdots,d\}$.
\end{itemize}
\end{proof}

Corollary \ref{spt-1} follows immediately from Theorem \ref{spt} (1) and the following Proposition \ref{1-closed}, in the proof of which a translation technique similar to the one in the proof of Proposition \ref{bounded-closed} (2) and a filling approximation argument are used.

\begin{proposition}\label{1-closed} Let $A_1,A_2,\cdots\subseteq\R$ be non-empty sets such that both $\min A_k$ and $\max A_k$ exist (not $\pm\infty$) for all $k\in\N$ and $\lim_{k\to\infty}(\max A_k-\min A_k)=0$.
\vspace{5pt}
\newline\emph{(1)} If $\sum_{k=1}^\infty\min A_k$ converges and $\sum_{k=1}^\infty\max A_k$ diverges, then $\sum_{k=1}^\infty A_k=\big[\sum_{k=1}^\infty\min A_k,+\infty\big)$.
\vspace{5pt}
\newline\emph{(2)} If $\sum_{k=1}^\infty\min A_k$ diverges and $\sum_{k=1}^\infty\max A_k$ converges, then $\sum_{k=1}^\infty A_k=\big(-\infty,\sum_{k=1}^\infty\max A_k\big]$.
\vspace{5pt}
\newline\emph{(3)} If both $\sum_{k=1}^\infty\min A_k$ and $\sum_{k=1}^\infty\max A_k$ diverge, then $\sum_{k=1}^\infty A_k=\emptyset$ or $\R$.
\end{proposition}
\begin{proof} (1) Suppose that $\sum_{k=1}^\infty\min A_k$ converges and $\sum_{k=1}^\infty\max A_k$ diverges. We need to prove $\sum_{k=1}^\infty A_k=[\sum_{k=1}^\infty\min A_k,+\infty)$. For each $k\in\N$, let $B_k:=A_k-\min A_k$. Then $\sum_{k=1}^\infty A_k=\sum_{k=1}^\infty B_k+\sum_{k=1}^\infty\min A_k$. It suffices to prove $\sum_{k=1}^\infty B_k=[0,+\infty)$. By $\min B_k=0$ for all $k\in\N$, we get $0\in\sum_{k=1}^\infty B_k\subseteq[0,+\infty)$. In the following we only need to prove $(0,+\infty)\subseteq\sum_{k=1}^\infty B_k$. For each $k\in\N$, let $b_k:=\max B_k=\max A_k-\min A_k\ge0$. Then $\lim_{k\to\infty}b_k=\lim_{k\to\infty}(\max A_k-\min A_k)=0$, but $\sum_{k=1}^\infty b_k=\sum_{k=1}^\infty(\max A_k-\min A_k)$ must diverge since $\sum_{k=1}^\infty\min A_k$ converges and $\sum_{k=1}^\infty\max A_k$ diverges. It follows from $b_k\ge0$ for all $k\in\N$ that $\sum_{k=1}^\infty b_k=+\infty$. Arbitrarily take $x\in(0,+\infty)$. We only need to prove $x\in\sum_{k=1}^\infty B_k$. Since $0,b_k\in B_k$ for all $k\in\N$, it suffices to find finitely or infinitely many positive integers $k_1<k_2<k_3<\cdots$ such that $x=\sum_{n}b_{k_n}$. By $b_1,b_2,\cdots\ge0$, $\lim_{k\to\infty}b_k=0$ and $\sum_{k=1}^\infty b_k=+\infty$, there exist integers $1\le p_1\le q_1$ such that
$$\sum_{k=p_1}^{q_1}b_k\le x<\sum_{k=p_1}^{q_1+1}b_k.$$
If $\sum_{k=p_1}^{q_1}b_k=x$, the proof is complete. Suppose $\sum_{k=p_1}^{q_1}b_k<x$ in the following. By $b_{q_1+2},b_{q_1+3},\cdots\ge0$, $\lim_{k\to\infty}b_k=0$ and $\sum_{k=q_1+2}^\infty b_k=+\infty$, there exist integers $q_1+2\le p_2\le q_2$ such that
$$\sum_{k=p_1}^{q_1}b_k+\sum_{k=p_2}^{q_2}b_k\le x<\sum_{k=p_1}^{q_1}b_k+\sum_{k=p_2}^{q_2+1}b_k.$$
If $\sum_{k=p_1}^{q_1}b_k+\sum_{k=p_2}^{q_2}b_k=x$, the proof is complete. Suppose $\sum_{k=p_1}^{q_1}b_k+\sum_{k=p_2}^{q_2}b_k<x$ in the following. $\cdots$ Repeat this process. If $\sum_{k=p_1}^{q_1}b_k+\sum_{k=p_2}^{q_2}b_k+\cdots+\sum_{k=p_n}^{q_n}b_k=x$ for finitely many positive integers $p_1\le q_1<q_1+1<p_2\le q_2<q_2+1<\cdots<p_n\le q_n$, the proof is complete. Otherwise, there exist infinitely many positive integers $p_1\le q_1<q_1+1<p_2\le q_2<q_2+1<p_3\le q_3<\cdots$ such that for all $n\in\N$ we have
$$\sum_{k=p_1}^{q_1}b_k+\cdots+\sum_{k=p_n}^{q_n}b_k<x<\sum_{k=p_1}^{q_1}b_k+\cdots+\sum_{k=p_n}^{q_n}b_k+b_{q_n+1}.$$
By $\lim_{n\to\infty} b_{q_n+1}=0$, we get $x=\sum_{n=1}^\infty\sum_{k=p_n}^{q_n}b_k$. Thus $x\in\sum_{k=1}^\infty B_k$.
\newpage\noindent(2) follows in a way similar to (1).
\vspace{5pt}
\newline(3) Suppose that both $\sum_{k=1}^\infty\min A_k$ and $\sum_{k=1}^\infty\max A_k$ diverge and $\sum_{k=1}^\infty A_k\neq\emptyset$. We need to prove $\sum_{k=1}^\infty A_k=\R$. By $\sum_{k=1}^\infty A_k\neq\emptyset$ there exist $a_1\in A_1,a_2\in A_2,\cdots$ such that $\sum_{k=1}^\infty a_k$ converges. For each $k\in\N$, let $B_k:=A_k-a_k$. Then $\sum_{k=1}^\infty A_k=\sum_{k=1}^\infty B_k+\sum_{k=1}^\infty a_k$. It suffices to prove $\sum_{k=1}^\infty B_k=\R$. Since the proofs of $(-\infty,0]\subseteq\sum_{k=1}^\infty B_k$ and $[0,+\infty)\subseteq\sum_{k=1}^\infty B_k$ are similar, we only prove $[0,+\infty)\subseteq\sum_{k=1}^\infty B_k$ in the following. In fact, for all $k\in\N$, by $0=a_k-a_k\in A_k-a_k=B_k$ we get $\min B_k\le0\le\max B_k$. It follows from $\lim_{k\to\infty}(\max B_k-\min B_k)=\lim_{k\to\infty}(\max A_k-\min A_k)=0$ that $\lim_{k\to\infty}\max B_k=0$. Since $\sum_{k=1}^\infty\max A_k$ diverges, $\sum_{k=1}^\infty a_k$ converges and $\max B_k=\max A_k-a_k$ for all $k\in\N$, we know that $\sum_{k=1}^\infty\max B_k$ must diverge. By $\max B_k\ge0$ for all $k\in\N$, we get $\sum_{k=1}^\infty\max B_k=+\infty$. For all $k\in\N$, let $b_k:=\max B_k\ge0$. Then $\lim_{k\to\infty}b_k=0$, $\sum_{k=1}^\infty b_k=+\infty$ and $0,b_k\in B_k$ for all $k\in\N$. In the same way as (1), we get $[0,+\infty)\subseteq\sum_{k=1}^\infty B_k$.
\end{proof}

Now we give an example for Remark \ref{spt-re} to end this section. We construct non-empty finite sets $A_1,A_2,\cdots\subseteq[0,1]$ such that $\delta_{A_1}*\delta_{A_2}*\cdots$ exists but $\sum_{k=1}^\infty A_k$ is not closed in the following. Let $\frac{1}{\N}:=\Big\{1,\frac{1}{2},\frac{1}{3},\frac{1}{4},\cdots\Big\}$ and use $\lfloor x\rfloor$ to denote the greatest integer no larger than $x$.

Define $r_1:=\frac{1}{2}$ and take $a_1\in(0,r_1)\setminus\frac{1}{\N}$.

Define $r_2:=\min\Big\{\frac{1}{2^2},\frac{1}{\lfloor\frac{1}{a_1}\rfloor}-a_1\Big\}>0$ and take $a_2\in(0,r_2)\setminus\frac{1}{\N}$.

Define $r_3:=\min\Big\{\frac{1}{2^3},\frac{1}{\lfloor\frac{1}{a_1}\rfloor}-a_1-a_2,\frac{1}{\lfloor\frac{1}{a_2}\rfloor}-a_2\Big\}>0$ and take $a_3\in(0,r_3)\setminus\frac{1}{\N}$.

$\cdots$

Define $r_k:=\min\Big\{\frac{1}{2^k},\frac{1}{\lfloor\frac{1}{a_1}\rfloor}-a_1-a_2-\cdots-a_{k-1},\frac{1}{\lfloor\frac{1}{a_2}\rfloor}-a_2-\cdots-a_{k-1},\cdots,\frac{1}{\lfloor\frac{1}{a_{k-1}}\rfloor}-a_{k-1}\Big\}>0$ and take $a_k\in(0,r_k)\setminus\frac{1}{\N}.$

$\cdots$

For all $k\in\N$, define $t_k:=a_k+a_{k+1}+a_{k+2}+\cdots<\sum_{n=k}^\infty\frac{1}{2^n}=\frac{1}{2^{k-1}}$. First we prove
\begin{equation}\label{ak>}
a_k>\frac{1}{\lfloor\frac{1}{t_k}\rfloor+1}\quad\text{for all }k\in\N.
\end{equation}
Arbitrarily take $k\in\N$. It suffices to show $\lfloor\frac{1}{t_k}\rfloor+1>\frac{1}{a_k}$. Since $\lfloor\frac{1}{a_k}\rfloor+1>\frac{1}{a_k}$, we only need to prove $\lfloor\frac{1}{t_k}\rfloor\ge\lfloor\frac{1}{a_k}\rfloor$. It suffices to show $\frac{1}{t_k}\ge\lfloor\frac{1}{a_k}\rfloor$. For all $n\in\N$, by
$$a_{k+n}<r_{k+n}\le\frac{1}{\lfloor\frac{1}{a_k}\rfloor}-a_k-a_{k+1}-\cdots-a_{k+n-1}$$
we get
$$a_k+a_{k+1}+\cdots+a_{k+n}<\frac{1}{\lfloor\frac{1}{a_k}\rfloor}.$$
As $n\to\infty$, it follows that $t_k\le\frac{1}{\lfloor\frac{1}{a_k}\rfloor}$, i.e., $\frac{1}{t_k}\ge\lfloor\frac{1}{a_k}\rfloor$. Therefore (\ref{ak>}) holds.

For all $k\in\N$, by (\ref{ak>}) we can take different
$$b_{k,1},b_{k,2},\cdots,b_{k,k^2}\in\Big(\frac{1}{\lfloor\frac{1}{t_k}\rfloor+1},a_k\Big)\setminus\frac{1}{\N}$$
and define
$$A_k:=\Big\{0,b_{k,1},b_{k,2},\cdots,b_{k,k^2},a_k,\frac{k}{k+1}\Big\}\subseteq[0,1].$$

The fact that $\delta_{A_1}*\delta_{A_2}*\cdots$ exists follows from Corollary \ref{exist-cor} and
$$\begin{aligned}
\sum_{k=1}^\infty\frac{1}{\#A_k}\sum_{a\in A_k}\frac{a}{1+a}&<\sum_{k=1}^\infty\frac{1}{k^2+3}\sum_{a\in A_k}a<\sum_{k=1}^\infty\frac{1}{k^2+3}\Big((k^2+1)a_k+\frac{k}{k+1}\Big)\\
&<\sum_{k=1}^\infty\Big(\frac{k^2+1}{k^2+3}\cdot\frac{1}{2^k}+\frac{k}{(k^2+3)(k+1)}\Big)<\sum_{k=1}^\infty(\frac{1}{2^k}+\frac{1}{k^2})<\infty.
\end{aligned}$$

In the following we only need to prove that $\sum_{k=1}^\infty A_k$ is not closed. By $\{0,\frac{n}{n+1}\}\subseteq A_n$ for all $n\in\N$, one can easily verify $\frac{n}{n+1}\in\sum_{k=1}^\infty A_k$ for all $n\in\N$. Thus $1=\lim_{n\to\infty}\frac{n}{n+1}\in\overline{\sum_{k=1}^\infty A_k}$. It suffices to show $1\notin\sum_{k=1}^\infty A_k$ in the following. By contradiction, we assume $1\in\sum_{k=1}^\infty A_k$. Then there exist $x_1\in A_1$, $x_2\in A_2$, $\cdots$ such that $\sum_{k=1}^\infty x_k=1$. If $x_k\le a_k$ for all $k\in\N$, then $\sum_{k=1}^\infty x_k\le\sum_{k=1}^\infty a_k<\sum_{k=1}^\infty r_k\le\sum_{k=1}^\infty\frac{1}{2^k}=1$, which contradicts $\sum_{k=1}^\infty x_k=1$. Thus there must exist $n\in\N$ such that $x_n>a_n$. By $x_n\in A_n$ we know $x_n=\frac{n}{n+1}$. Thus
\begin{equation}\label{=1}
x_1+\cdots+x_{n-1}+\frac{n}{n+1}+x_{n+1}+x_{n+2}+\cdots=1.
\end{equation}
If there exists $n'\in\N\setminus\{n\}$ such that $x_{n'}>a_{n'}$, by $x_{n'}\in A_{n'}$ we know $x_{n'}=\frac{n'}{n'+1}$ and then $\sum_{k=1}^\infty x_k\ge\frac{n}{n+1}+\frac{n'}{n'+1}>1$, which contradicts $\sum_{k=1}^\infty x_k=1$. Thus $x_k\le a_k$ for all $k\in\N\setminus\{n\}$. If $x_1=x_2=\cdots=x_{n-1}=0$, then
$$x_1+\cdots+x_{n-1}+\frac{n}{n+1}+x_{n+1}+x_{n+2}+\cdots\le\frac{n}{n+1}+a_{n+1}+a_{n+2}+\cdots$$
$$<\frac{n}{n+1}+\frac{1}{2^{n+1}}+\frac{1}{2^{n+2}}+\cdots=\frac{n}{n+1}+\frac{1}{2^n}\le1,$$
which contradicts (\ref{=1}). Thus there must exist $p\in\{1,2,\cdots,n-1\}$ such that $x_1=\cdots=x_{p-1}=0$ and $x_p>0$. We get
$$t_p=a_p+a_{p+1}+a_{p+2}+\cdots>x_1+\cdots+x_{n-1}+x_{n+1}+x_{n+2}+\cdots\xlongequal[]{\text{by (\ref{=1})}}\frac{1}{n+1},$$
which implies
\begin{equation}\label{n+1>}
n+1>\frac{1}{t_p}\ge\lfloor\frac{1}{t_p}\rfloor.
\end{equation}
Since $x_p>0$ and $x_p\in A_p$ imply $x_p>\frac{1}{\lfloor\frac{1}{t_p}\rfloor+1}$, we get $\lfloor\frac{1}{t_p}\rfloor+1>\frac{1}{x_p}\ge\lfloor\frac{1}{x_p}\rfloor$ and then $\lfloor\frac{1}{t_p}\rfloor\ge\lfloor\frac{1}{x_p}\rfloor$. It follows from (\ref{n+1>}) that $n+1>\lfloor\frac{1}{x_p}\rfloor$, which implies $n+1\ge\lfloor\frac{1}{x_p}\rfloor+1>\frac{1}{x_p}$. Therefore $x_p>\frac{1}{n+1}$ and then $x_1+\cdots+x_{n-1}+\frac{n}{n+1}+x_{n+1}+x_{n+2}+\cdots>\frac{1}{n+1}+\frac{n}{n+1}=1$. This contradicts (\ref{=1}).

\section{Proofs of Theorem \ref{sum union} and Corollary \ref{dim-sum}}
\indent

First we prove the following decomposition of $\sum_{k=1}^\infty(A_k\cup A_k')$.

\begin{proposition}\label{decom}
For each $k\in\N$, let $A_k,A_k'\subseteq\R^d$ where $A_k$ is non-empty and $A_k'$ may be empty. If $\varliminf_{k\to\infty}\inf_{a\in A_k'}|a|>0$, then
$$\sum_{k=1}^\infty\big(A_k\cup A_k'\big)=\bigcup_{p=1}^\infty\Big(\sum_{k=1}^p\big(A_k\cup A_k'\big)+\sum_{k=p+1}^\infty A_k\Big).$$
\end{proposition}
\begin{proof}
Since the inclusion ``$\supseteq$'' is obvious, we only prove ``$\subseteq$'' in the following. Let $x\in\sum_{k=1}^\infty(A_k\cup A_k')$. Then there exist $x^{(1)}\in A_1\cup A_1'$, $x^{(2)}\in A_2\cup A_2'$, $\cdots$ such that $x=\sum_{k=1}^\infty x^{(k)}$ converges, which implies $\lim_{k\to\infty}|x^{(k)}|=0$. Let $c\in(0,\varliminf_{k\to\infty}\inf_{a\in A_k'}|a|)$. Then there exists $p_1\in\N$ such that for all $k>p_1$ we have $\inf_{a\in A_k'}|a|>c$. By $\lim_{k\to\infty}|x^{(k)}|=0$, there exists $p_2\in\N$ such that for all $k>p_2$ we have $|x^{(k)}|<c$. Let $p_0:=\max\{p_1,p_2\}$. Then for all $k\ge p_0+1$ we have $|x^{(k)}|<\inf_{a\in A_k'}|a|$, which implies $x^{(k)}\notin A_k'$. It follows from $x^{(k)}\in A_k\cup A_k'$ that $x^{(k)}\in A_k$ for all $k\ge p_0+1$. Thus
$$x=\sum_{k=1}^{p_0}x^{(k)}+\sum_{k=p_0+1}^\infty x^{(k)}\in\sum_{k=1}^{p_0}\big(A_k\cup A_k'\big)+\sum_{k=p_0+1}^\infty A_k.$$
We get $x\in\bigcup_{p=1}^\infty\big(\sum_{k=1}^p(A_k\cup A_k')+\sum_{k=p+1}^\infty A_k\big)$.
\end{proof}

Except for Proposition \ref{decom}, we need the following.

\begin{proposition}\label{closed+closed}
Let $n\in\N$ and $A_1,A_2,\cdots,A_n\subseteq\R^d$ be closed sets. If for every $j\in\{1,\cdots,d\}$ we have
\begin{equation}\label{min or max n}
\min_{1\le k\le n}\inf(A_k)_j>-\infty\quad\text{or}\quad\max_{1\le k\le n}\sup(A_k)_j<+\infty,
\end{equation}
then $A_1+A_2+\cdots+A_n$ is closed.
\end{proposition}
\begin{proof} (1) First we prove that for any closed sets $X,Y\subseteq\R^d$ with
\begin{equation}\label{min or max 2}
\min\big\{\inf(X)_j,\inf(Y)_j\big\}>-\infty\quad\text{or}\quad\max\big\{\sup(X)_j,\sup(Y)_j\big\}<+\infty
\end{equation}
for every $j\in\{1,\cdots,d\}$, the sum $X+Y$ is closed.

Let $z^{(1)},z^{(2)},\cdots\in X+Y$ such that $z^{(k)}=(z^{(k)}_1,\cdots,z^{(k)}_d)$ converges to some $z=(z_1,\cdots,z_d)\in\R^d$. We need to prove $z\in X+Y$. For each $k\in\N$, by $z^{(k)}\in X+Y$, there exist $x^{(k)}=(x^{(k)}_1,\cdots,x^{(k)}_d)\in X$ and $y^{(k)}=(y^{(k)}_1,\cdots,y^{(k)}_d)\in Y$ such that $z^{(k)}=x^{(k)}+y^{(k)}$. For all $j\in\{1,\cdots,d\}$, we have
$$\lim_{k\to\infty}(x^{(k)}_j+y^{(k)}_j)=\lim_{k\to\infty}z^{(k)}_j=z_j.$$

If $\varliminf_{k\to\infty}x^{(k)}_1=-\infty$, then $\varlimsup_{k\to\infty}y^{(k)}_1=+\infty$, $\inf(X)_1=-\infty$ and $\sup(Y)_1=+\infty$, which contradict (\ref{min or max 2}).

If $\varliminf_{k\to\infty}x^{(k)}_1=+\infty$, then $\varlimsup_{k\to\infty}y^{(k)}_1=-\infty$, $\sup(X)_1=+\infty$ and $\inf(Y)_1=-\infty$, which contradict (\ref{min or max 2}).

Thus we must have $\varliminf_{k\to\infty}x^{(k)}_1\in(-\infty,+\infty)$. There exists a subsequence $\{k_p\}_{p\ge1}$ of $\N$ such that $\lim_{p\to\infty}x^{(k_p)}_1=x_1$ for some $x_1\in\R$.

If $\varliminf_{p\to\infty}x^{(k_p)}_2=-\infty$, then $\varlimsup_{p\to\infty}y^{(k_p)}_2=+\infty$, $\inf(X)_2=-\infty$ and $\sup(Y)_2=+\infty$, which contradict (\ref{min or max 2}).

If $\varliminf_{p\to\infty}x^{(k_p)}_2=+\infty$, then $\varlimsup_{p\to\infty}y^{(k_p)}_2=-\infty$, $\sup(X)_2=+\infty$ and $\inf(Y)_2=-\infty$, which contradict (\ref{min or max 2}).

Thus we must have $\varliminf_{p\to\infty}x^{(k_p)}_2\in(-\infty,+\infty)$. There exists a subsequence $\{k_{p_q}\}_{q\ge1}$ of $\{k_p\}_{p\ge1}$ such that $\lim_{q\to\infty}x^{(k_{p_q})}_2=x_2$ for some $x_2\in\R$, and we also have $\lim_{q\to\infty}x^{(k_{p_q})}_1=x_1$.

$\cdots$ Repeat this process $d$ times. Finally we get a subsequence of positive integers $r_1,r_2,\cdots$ such that $\lim_{n\to\infty}x^{(r_n)}_j=x_j$ for some $x_j\in\R$ for every $j\in\{1,\cdots,d\}$. Let $x:=(x_1,\cdots,x_d)$ and $y:=z-x$. Then $\lim_{n\to\infty}x^{(r_n)}=x$ and $\lim_{n\to\infty}y^{(r_n)}=\lim_{n\to\infty}(z^{(r_n)}-x^{(r_n)})=z-x=y$. Since $X$ and $Y$ are both closed, $\{x^{(r_n)}\}_{n\ge1}\subseteq X$ and $\{y^{(r_n)}\}_{n\ge1}\subseteq Y$, we get $x\in X$ and $y\in Y$, and then $z=x+y\in X+Y$.
\vspace{5pt}
\newline(2) Now we prove this proposition by induction on $n$.
\newline For $n=1$, this proposition obviously holds. Assume that this proposition holds for some $n\in\N$. Let $A_1,A_2,\cdots,A_n,A_{n+1}\subseteq\R^d$ be closed sets such that for every $j\in\{1,\cdots,d\}$ we have
\begin{equation}\label{min or max n+1}
\min_{1\le k\le n+1}\inf(A_k)_j>-\infty\quad\text{or}\quad\max_{1\le k\le n+1}\sup(A_k)_j<+\infty.
\end{equation}
Then (\ref{min or max n}) holds. By the induction hypothesis, we know that $A_1+\cdots+A_n$ is closed. In the following we only need to prove that $A_1+\cdots+A_n+A_{n+1}$ is closed by using (1). For every $j\in\{1,\cdots,d\}$, noting (\ref{min or max n+1}):

if $\min_{1\le k\le n+1}\inf(A_k)_j>-\infty$, then
$$\min\big\{\inf(A_1+\cdots+A_n)_j,\inf(A_{n+1})_j\big\}\ge\min\big\{\inf(A_1)_j+\cdots+\inf(A_n)_j,\inf(A_{n+1})_j\big\}>-\infty;$$

if $\max_{1\le k\le n+1}\sup(A_k)_j<+\infty$, then
$$\max\big\{\sup(A_1+\cdots+A_n)_j,\sup(A_{n+1})_j\big\}\le\max\big\{\sup(A_1)_j+\cdots+\sup(A_n)_j,\sup(A_{n+1})_j\big\}<+\infty.$$
Let $X:=A_1+\cdots+A_n$ and $Y:=A_{n+1}$. By (1) we know that $X+Y=A_1+\cdots+A_n+A_{n+1}$ is closed.
\end{proof}

The condition (\ref{min or max n}) in Proposition \ref{closed+closed} can not be omitted. Otherwise, we can take $d=1$, $n=2$, $A_1=\big\{k+\frac{1}{2^k}:k\in\N\big\}$ and $A_2=\big\{-k-\frac{1}{2^{k+1}}:k\in\N\big\}$. Then $A_1,A_2$ are both closed. But $A_1+A_2$ is not closed, since one can easily verify $0\in\overline{(A_1+A_2)}\setminus(A_1+A_2)$.

Now we use Propositions \ref{decom} and \ref{closed+closed} to prove Theorem \ref{sum union}. Except for a translation technique similar to the one in the proof of Proposition \ref{bounded-closed} (2), the proof of Theorem \ref{sum union} (1) relies on technical estimations on the absolute values of specific sums of the coordinate components of certain summable points.

\begin{proof}[Proof of Theorem \ref{sum union}] For $k\in\N$, let $A_k,A_k'\subseteq\R^d$ where $A_k$ is non-empty and $A_k'$ may be empty.
\vspace{5pt}
\newline(1) Suppose that $A_k\cup A_k'$ is closed for every $k\in\N$, both (\ref{inf-sup con}) and (\ref{inf-sup or}) hold for every $j\in\{1,\cdots,d\}$, and $\lim_{k\to\infty}\inf_{a\in A_k'}|a|=+\infty$. We need to prove that $\sum_{k=1}^\infty(A_k\cup A_k')$ is closed. For each $k\in\N$, define the translations of $A_k$ and $A_k'$ respectively by
$$B_k:=A_k-\big(\inf(A_k)_1,\cdots,\inf(A_k)_d\big)\quad\quad\quad\quad\quad\quad\quad\quad\quad\quad\quad\quad\quad$$
\begin{equation}\label{contained in}
\subseteq\big[0,\sup(A_k)_1-\inf(A_k)_1\big]\times\cdots\times\big[0,\sup(A_k)_d-\inf(A_k)_d\big]
\end{equation}
and
$$B_k':=A_k'-\big(\inf(A_k)_1,\cdots,\inf(A_k)_d\big).\quad\quad\quad\quad\quad\quad\quad\quad\quad\quad\quad\quad\quad$$
Then $B_k\cup B_k'$ is closed for every $k\in\N$, and for every $j\in\{1,\cdots,d\}$
\begin{equation}\label{B inf-sup con}
\text{both }\sum_{k=1}^\infty\inf(B_k)_j=0\text{ and }\sum_{k=1}^\infty\sup(B_k)_j=\sum_{k=1}^\infty\Big(\sup(A_k)_j-\inf(A_k)_j\Big)\text{ converge.}
\end{equation}
Besides, for each $j\in\{1,\cdots,d\}$, since (\ref{inf-sup con}) implies that $\big\{\inf(A_k)_j\big\}_{k\ge1}$ is bounded, by
$$\left\{\begin{aligned}
\inf_{k\in\N}\inf(B_k')_j&=\inf_{k\in\N}\Big(\inf(A_k')_j-\inf(A_k)_j\Big)\ge\inf_{k\in\N}\inf(A_k')_j-\sup_{k\in\N}\inf(A_k)_j\\
\sup_{k\in\N}\sup(B_k')_j&=\sup_{k\in\N}\Big(\sup(A_k')_j-\inf(A_k)_j\Big)\le\sup_{k\in\N}\sup(A_k')_j-\inf_{k\in\N}\inf(A_k)_j
\end{aligned}\right.$$
and (\ref{inf-sup or}), we get
\begin{equation}\label{B inf-sup or}
\inf_{k\in\N}\inf(B_k')_j>-\infty\quad\text{or}\quad\sup_{k\in\N}\sup(B_k')_j<+\infty.
\end{equation}
Moreover, since (\ref{inf-sup con}) implies that $\lim_{k\to\infty}\inf(A_k)_j=0$ for every $j\in\{1,\cdots,d\}$, by
\begin{align*}
\lim_{k\to\infty}\inf_{b\in B_k'}|b|&=\lim_{k\to\infty}\inf_{a\in A_k'}\Big|a-\big(\inf(A_k)_1,\cdots,\inf(A_k)_d\big)\Big|\\
&\ge\lim_{k\to\infty}\Big(\inf_{a\in A_k'}|a|-\Big|\big(\inf(A_k)_1,\cdots,\inf(A_k)_d\big)\Big|\Big)
\end{align*}
and $\lim_{k\to\infty}\inf_{a\in A_k'}|a|=+\infty$, we get
\begin{equation}\label{Bk' infty}
\lim_{k\to\infty}\inf_{b\in B_k'}|b|=+\infty.
\end{equation}
Note that
$$\sum_{k=1}^\infty\big(B_k\cup B_k'\big)=\sum_{k=1}^\infty\big(A_k\cup A_k'\big)-\Big(\sum_{k=1}^\infty\inf(A_k)_1,\cdots,\sum_{k=1}^\infty\inf(A_k)_d\Big).$$
We only need to prove that $\sum_{k=1}^\infty(B_k\cup B_k')$ is closed.
\newline\textcircled{\footnotesize{$1$}} First we prove that there exists $r\in\N$ such that for all $p>r$, the set $\sum_{k=1}^p(B_k\cup B_k')+\sum_{k=p+1}^\infty B_k$ is closed.

For every $j\in\{1,\cdots,d\}$, it follows from (\ref{B inf-sup con}) that $\lim_{k\to\infty}\inf(B_k)_j=\lim_{k\to\infty}\sup(B_k)_j=0$, and then there exists $r_1\in\N$ such that
\begin{equation}\label{Bk subset}
\text{for all }k>r_1\text{ we have }B_k\subseteq[-1,1]^d.
\end{equation}
Besides, by (\ref{Bk' infty}) there exists $r_2\in\N$ such that
\begin{equation}\label{Bk' cap}
\text{for all }k>r_2\text{ we have }B_k'\cap[-1,1]^d=\emptyset.
\end{equation}
Let $r:=\max\{r_1,r_2\}$ and arbitrarily take $p>r$. We only need to prove that $\sum_{k=1}^p(B_k\cup B_k')+\sum_{k=p+1}^\infty B_k$ is closed in the following.
\begin{itemize}
\item[i)] Prove that $\sum_{k=p+1}^\infty B_k$ is closed.
\newline In fact, for every $k\ge p+1$, by (\ref{Bk subset}) and (\ref{Bk' cap}) we get $B_k=(B_k\cup B_k')\cap[-1,1]^d$, where $B_k\cup B_k'$ and $[-1,1]^d$ are both closed. Thus $B_k$ is closed for every $k\ge p+1$. It follows from Proposition \ref{bounded-closed} (2) and (\ref{B inf-sup con}) that $\sum_{k=p+1}^\infty B_k$ is closed.
\item[ii)] Prove that $\sum_{k=1}^p(B_k\cup B_k')+\sum_{k=p+1}^\infty B_k$ is closed.
\newline Since $B_k\cup B_k'$ is closed for every $k\in\{1,\cdots,p\}$ and i) says that $\sum_{k=p+1}^\infty B_k$ is closed, by Proposition \ref{closed+closed}, it suffices to show that for every $j\in\{1,\cdots,d\}$ we have
$$\min_{1\le k\le p}\inf\big(B_k\cup B_k'\big)_j>-\infty\quad\text{and}\quad\inf\Big(\sum_{k=p+1}^\infty B_k\Big)_j>-\infty$$
or
$$\max_{1\le k\le p}\sup\big(B_k\cup B_k'\big)_j<+\infty\quad\text{and}\quad\sup\Big(\sum_{k=p+1}^\infty B_k\Big)_j<+\infty.$$
Arbitrarily take $j\in\{1,\cdots,d\}$. Since (\ref{B inf-sup con}) implies that
$$\text{both }\sum_{k=p+1}^\infty\inf(B_k)_j\text{ and }\sum_{k=p+1}^\infty\sup(B_k)_j\text{ converge,}$$
we get
$$\inf\Big(\sum_{k=p+1}^\infty B_k\Big)_j\ge\sum_{k=p+1}^\infty\inf(B_k)_j>-\infty\quad\text{and}\quad\sup\Big(\sum_{k=p+1}^\infty B_k\Big)_j\le\sum_{k=p+1}^\infty\sup(B_k)_j<+\infty.$$
In the following we only need to verify
$$\min_{1\le k\le p}\inf(B_k\cup B_k')_j>-\infty\quad\text{or}\quad\max_{1\le k\le p}\sup(B_k\cup B_k')_j<+\infty.$$
In fact it follows from (\ref{B inf-sup con}) that for all $1\le k\le p$ we have $\inf(B_k)_j>-\infty$ and $\sup(B_k)_j<+\infty$. By (\ref{B inf-sup or}) we get
$$\min_{1\le k\le p}\inf(B_k\cup B_k')_j=\min_{1\le k\le p}\min\Big\{\inf(B_k)_j,\inf(B_k')_j\Big\}>-\infty$$
or
$$\max_{1\le k\le p}\sup(B_k\cup B_k')_j=\max_{1\le k\le p}\max\Big\{\sup(B_k)_j,\sup(B_k')_j\Big\}<+\infty.$$
\end{itemize}
\textcircled{\footnotesize{$2$}} Now we prove that $\sum_{k=1}^\infty(B_k\cup B_k')$ is closed.
\newline Let $x^{(1)},x^{(2)},\cdots\in\sum_{k=1}^\infty(B_k\cup B_k')$ such that $x^{(n)}$ converges to some $x\in\R^d$. It suffices to show $x\in\sum_{k=1}^\infty(B_k\cup B_k')$. By Proposition \ref{decom} we only need to prove $x\in\sum_{k=1}^p(B_k\cup B_k')+\sum_{k=p+1}^\infty B_k$ for some $p\in\N$. Let $r\in\N$ be defined as in \textcircled{\footnotesize{$1$}}. Since \textcircled{\footnotesize{$1$}} says that $\sum_{k=1}^p(B_k\cup B_k')+\sum_{k=p+1}^\infty B_k$ is closed for all $p>r$, by $\lim_{n\to\infty}x^{(n)}=x$, it suffices to show $\{x^{(n)}\}_{n\ge1}\subseteq\sum_{k=1}^p(B_k\cup B_k')+\sum_{k=p+1}^\infty B_k$ for some $p>r$ in the following.

In fact, since $\{x^{(n)}\}_{n\ge1}$ is bounded, there exists $C>0$ such that
$$|x^{(n)}_1|+\cdots+|x^{(n)}_d|<C\quad\text{for all }n\in\N.$$
Let
$$M:=\sum_{j=1}^d\sum_{k=1}^\infty\big(\sup(A_k)_j-\inf(A_k)_j\big)<\infty.$$
If $B_k'=\emptyset$ for all $k\in\N$, it follows immediately from Proposition \ref{bounded-closed} (2) that $\sum_{k=1}^\infty(B_k\cup B_k')$ is closed, and the proof is complete. In the following we suppose $B_k'\neq\emptyset$ for some $k\in\N$. Then
$$\inf_{k\in\N}\inf(B_k')_j\neq+\infty\quad\text{and}\quad\sup_{k\in\N}\sup(B_k')_j\neq-\infty$$
for every $j\in\{1,\cdots,d\}$, and by (\ref{B inf-sup or}) we can define $m_j\in[0,+\infty)$ by
$$m_j:=\left\{\begin{array}{lll}
\max\Big\{\Big|\inf\limits_{k\in\N}\inf(B_k')_j\Big|,\Big|\sup\limits_{k\in\N}\sup(B_k')_j\Big|\Big\},&\text{if }\inf\limits_{k\in\N}\inf(B_k')_j>-\infty\text{ and }\sup\limits_{k\in\N}\sup(B_k')_j<+\infty;\\
\Big|\inf\limits_{k\in\N}\inf(B_k')_j\Big|,&\text{if }\inf\limits_{k\in\N}\inf(B_k')_j>-\infty\text{ and }\sup\limits_{k\in\N}\sup(B_k')_j=+\infty;\\
\Big|\sup\limits_{k\in\N}\sup(B_k')_j\Big|,&\text{if }\inf\limits_{k\in\N}\inf(B_k')_j=-\infty\text{ and }\sup\limits_{k\in\N}\sup(B_k')_j<+\infty.
\end{array}\right.$$
In the following, for all $a\in\R^d$, we use $a_j$ to denote the $j$-th coordinate of $a$. By
$$\lim_{k\to\infty}\inf_{b\in B_k'}\big(|b_1|+\cdots+|b_d|\big)\ge\lim_{k\to\infty}\inf_{b\in B_k'}|b|\xlongequal[\text{(\ref{Bk' infty})}]{\text{by}}+\infty,$$
there exists $s>r$ such that for all $k>s$ we have
\begin{equation}\label{>}
\inf_{b\in B_k'}\big(|b_1|+\cdots+|b_d|\big)>2\sum_{j=1}^dm_j,
\end{equation}
and there exists $p>s$ such that for all $k>p$ we have
\begin{equation}\label{>++}
\inf_{b\in B_k'}\big(|b_1|+\cdots+|b_d|\big)>C+M+(2+2s)\sum_{j=1}^dm_j.
\end{equation}
Arbitrarily take $n\in\N$. It suffices to prove $x^{(n)}\in\sum_{k=1}^p(B_k\cup B_k')+\sum_{k=p+1}^\infty B_k$. By $x^{(n)}\in\sum_{k=1}^\infty(B_k\cup B_k')$, there exist $x^{(n,1)}\in B_1\cup B_1'$, $x^{(n,2)}\in B_2\cup B_2'$, $\cdots$ such that $x^{(n)}=\sum_{k=1}^\infty x^{(n,k)}$ converges. Arbitrarily take $t\ge p+1$. We only need to prove $x^{(n,t)}\in B_t$. By contradiction, assume $x^{(n,t)}\notin B_t$. Then $x^{(n,t)}\in B_t'$. Define
$$K'_1:=\{k\le s:x^{(n,k)}\in B_k'\},\quad K'_2:=\{k\ge s+1:x^{(n,k)}\in B_k'\},$$
$$K':=K'_1\cup K'_2\quad\text{and}\quad K:=\N\setminus K'.$$
Then $t\in K'_2\subseteq K'$. Since the convergence of $\sum_{k=1}^\infty x^{(n,k)}$ implies $\lim_{k\to\infty}|x^{(n,k)}|=0$, by (\ref{Bk' infty}) we get $\#K'<\infty$. For all $j\in\{1,\cdots,d\}$, we have
$$x^{(n)}_j=\sum_{k\in K}x^{(n,k)}_j+\sum_{k\in K'}x^{(n,k)}_j.$$
i) On the one hand, we have
$$\begin{aligned}
\sum_{j=1}^d\Big|\sum_{k\in K'}x^{(n,k)}_j\Big|&=\sum_{j=1}^d\Big|x^{(n)}_j-\sum_{k\in K}x^{(n,k)}_j\Big|\le\sum_{j=1}^d|x^{(n)}_j|+\sum_{j=1}^d\sum_{k\in K}|x^{(n,k)}_j|\\
&<C+\sum_{j=1}^d\sum_{\substack{k\ge1\\x^{(n,k)}\in B_k}}|x^{(n,k)}_j|\overset{\text{by (\ref{contained in})}}{\le}C+M.
\end{aligned}$$
ii) On the other hand, we can prove
$$\sum_{j=1}^d\Big|\sum_{k\in K'}x^{(n,k)}_j\Big|\ge\sum_{j=1}^d|x^{(n,t)}_j|-(2+2s)\sum_{j=1}^dm_j$$
as follows. Let
$$E:=\Big\{j\in\{1,\cdots,d\}:\inf_{k\in\N}\inf(B_k')_j>-\infty\Big\}\quad\text{and}\quad F:=\big\{1,\cdots,d\big\}\setminus E.$$
Since
\begin{align*}
\sum_{j=1}^d\Big|\sum_{k\in K'}x^{(n,k)}_j\Big|&=\sum_{j\in E}\Big|\sum_{\substack{k\in K'\\x^{(n,k)}_j\ge0}}x^{(n,k)}_j+\sum_{\substack{k\in K'\\x^{(n,k)}_j<0}}x^{(n,k)}_j\Big|+\sum_{j\in F}\Big|\sum_{\substack{k\in K'\\x^{(n,k)}_j<0}}x^{(n,k)}_j+\sum_{\substack{k\in K'\\x^{(n,k)}_j\ge0}}x^{(n,k)}_j\Big|\\
&\ge\sum_{j\in E}\Big(\sum_{\substack{k\in K'\\x^{(n,k)}_j\ge0}}x^{(n,k)}_j-\sum_{\substack{k\in K'\\x^{(n,k)}_j<0}}(-x^{(n,k)}_j)\Big)+\sum_{j\in F}\Big(\sum_{\substack{k\in K'\\x^{(n,k)}_j<0}}(-x^{(n,k)}_j)-\sum_{\substack{k\in K'\\x^{(n,k)}_j\ge0}}x^{(n,k)}_j\Big)\\
&=\sum_{j\in E}\Big(\sum_{\substack{k\in K'\\x^{(n,k)}_j\ge0}}|x^{(n,k)}_j|-\sum_{\substack{k\in K'\\x^{(n,k)}_j<0}}|x^{(n,k)}_j|\Big)+\sum_{j\in F}\Big(\sum_{\substack{k\in K'\\x^{(n,k)}_j<0}}|x^{(n,k)}_j|-\sum_{\substack{k\in K'\\x^{(n,k)}_j\ge0}}|x^{(n,k)}_j|\Big)\\
&=\sum_{j\in E}\Big(\sum_{k\in K'}|x^{(n,k)}_j|-2\sum_{\substack{k\in K'\\x^{(n,k)}_j<0}}|x^{(n,k)}_j|\Big)+\sum_{j\in F}\Big(\sum_{k\in K'}|x^{(n,k)}_j|-2\sum_{\substack{k\in K'\\x^{(n,k)}_j\ge0}}|x^{(n,k)}_j|\Big)\\
&=\sum_{k\in K'}\sum_{j=1}^d|x^{(n,k)}_j|-2\sum_{j\in E}\sum_{\substack{k\in K'\\x^{(n,k)}_j<0}}|x^{(n,k)}_j|-2\sum_{j\in F}\sum_{\substack{k\in K'\\x^{(n,k)}_j\ge0}}|x^{(n,k)}_j|,
\end{align*}
noting $t\in K'_2\subseteq K'$, we only need to prove
$$\sum_{k\in K'\setminus\{t\}}\sum_{j=1}^d|x^{(n,k)}_j|+(2+2s)\sum_{j=1}^dm_j\ge2\sum_{j\in E}\sum_{\substack{k\in K'\\x^{(n,k)}_j<0}}|x^{(n,k)}_j|+2\sum_{j\in F}\sum_{\substack{k\in K'\\x^{(n,k)}_j\ge0}}|x^{(n,k)}_j|.$$
It suffices to combine the following \textcircled{\footnotesize{$a$}}, \textcircled{\footnotesize{$b$}} and \textcircled{\footnotesize{$c$}}.
\begin{itemize}
\item[\textcircled{\footnotesize{$a$}}] We have
$$\sum_{k\in K'\setminus\{t\}}\sum_{j=1}^d|x^{(n,k)}_j|+(2+2s)\sum_{j=1}^dm_j\ge2(\#K')\sum_{j=1}^dm_j,$$
since
\begin{align*}
\sum_{k\in K'\setminus\{t\}}\sum_{j=1}^d|x^{(n,k)}_j|&\ge\sum_{k\in K'_2\setminus\{t\}}\Big(\sum_{j=1}^d|x^{(n,k)}_j|\Big)\overset{\text{by (\ref{>})}}{\ge}\sum_{k\in K'_2\setminus\{t\}}\Big(2\sum_{j=1}^dm_j\Big)=2(\#K'_2-1)\sum_{j=1}^dm_j\\
&=2(\#K'-\#K'_1-1)\sum_{j=1}^dm_j\ge2(\#K')\sum_{j=1}^dm_j-(2+2s)\sum_{j=1}^dm_j.
\end{align*}
\item[\textcircled{\footnotesize{$b$}}] We have
$$\sum_{j\in E}\sum_{\substack{k\in K'\\x^{(n,k)}_j<0}}|x^{(n,k)}_j|\overset{(\star)}{\le}\sum_{j\in E}\sum_{\substack{k\in K'\\x^{(n,k)}_j<0}}m_j\le(\#K')\sum_{j\in E}m_j,$$
where ($\star$) follows from the fact that for all $j\in E$ and $k\in K'$ with $x^{(n,k)}_j<0$, we can prove $|x^{(n,k)}_j|\le m_j$. In fact, by $j\in E$ we get
$$\inf_{l\in\N}\inf(B_l')_j>-\infty.$$
It follows from the definition of $m_j$ that
$$m_j\ge\Big|\inf_{l\in\N}\inf(B_l')_j\Big|.$$
Besides, by $k\in K'$ we get $x^{(n,k)}\in B_k'$ and then
$$-\infty<\inf_{l\in\N}\inf(B_l')_j\le\inf(B_k')_j\le x^{(n,k)}_j.$$
It follows from $x^{(n,k)}_j<0$ that
$$|x^{(n,k)}_j|\le\Big|\inf_{l\in\N}\inf(B_l')_j\Big|\le m_j.$$
\item[\textcircled{\footnotesize{$c$}}] We have
$$\sum_{j\in F}\sum_{\substack{k\in K'\\x^{(n,k)}_j\ge0}}|x^{(n,k)}_j|\overset{(\star\star)}{\le}\sum_{j\in F}\sum_{\substack{k\in K'\\x^{(n,k)}_j\ge0}}m_j\le(\#K')\sum_{j\in F}m_j,$$
where ($\star\star$) follows from a way similar to ($\star$) in the above \textcircled{\footnotesize{$b$}} noting (\ref{B inf-sup or}).
\end{itemize}
Combing i) and ii) we get
$$\sum_{j=1}^d|x^{(n,t)}_j|-(2+2s)\sum_{j=1}^dm_j<C+M,$$
which contradicts $t\ge p+1$, $x^{(n,t)}\in B_t'$ and (\ref{>++}).
\vspace{5pt}
\newline(2) Suppose that $A_k'$ is at most countable for every $k\in\N$ and $\varliminf_{k\to\infty}\inf_{a\in A_k'}|a|>0$. Since the proofs of \textcircled{\footnotesize{$1$}} and \textcircled{\footnotesize{$2$}} in Theorem \ref{sum union} (2) are similar, we only prove \textcircled{\footnotesize{$1$}} in the following. By the definition of Hausdorff dimension, it suffices to show that
$$\cH^s\Big(\sum_{k=1}^\infty\big(A_k\cup A_k'\big)\Big)=0\text{ if and only if }\cH^s\Big(\sum_{k=1}^\infty A_k\Big)=0\quad\text{for all }s\in[0,d].$$
\boxed{\Rightarrow} follows immediately from $\sum_{k=1}^\infty(A_k\cup A_k')\supseteq\sum_{k=1}^\infty A_k$.
\newline\boxed{\Leftarrow} Suppose $\cH^s\big(\sum_{k=1}^\infty A_k\big)=0$ for some $s\in[0,d]$. By Proposition \ref{decom} we only need to prove $\cH^s\big(\sum_{k=1}^p(A_k\cup A_k')+\sum_{k=p+1}^\infty A_k\big)=0$ for all $p\in\N$. Since
$$\sum_{k=1}^p(A_k\cup A_k')+\sum_{k=p+1}^\infty A_k=\bigcup_{\substack{D\cup D'=\{1,\cdots,p\}\\ D\cap D'=\emptyset}}\Big(\sum_{k\in D}A_k+\sum_{k\in D'}A_k'+\sum_{k=p+1}^\infty A_k\Big)$$
is a finite union, it suffices to show
\begin{equation}\label{=0}
\cH^s\Big(\sum_{k\in D}A_k+\sum_{k\in D'}A_k'+\sum_{k=p+1}^\infty A_k\Big)=0
\end{equation}
for all $D$ and $D'$ with $D\cup D'=\{1,\cdots,p\}$ and $D\cap D'=\emptyset$ in the following. Let $a\in\sum_{k\in D'}A_k$. By
$$\sum_{k\in D}A_k+\sum_{k\in D'}A_k'+\sum_{k=p+1}^\infty A_k=\bigcup_{x\in\sum_{k\in D'}A_k'}\Big(x-a+\sum_{k\in D}A_k+a+\sum_{k=p+1}^\infty A_k\Big)\subseteq\bigcup_{x\in\sum_{k\in D'}A_k'}\Big(x-a+\sum_{k=1}^\infty A_k\Big),$$
where $\sum_{k\in D'}A_k'$ is at most countable and $\cH^s\big(\sum_{k=1}^\infty A_k\big)=0$, we get (\ref{=0}).
\end{proof}

Now we prove Corollary \ref{dim-sum} to end this section.

\begin{proof}[Proof of Corollary \ref{dim-sum}] For each $k\in\N$, let $c_k\ge1$ and $C_k\ge c_k+1$ be real numbers, $B_k\subseteq\R^d$, $G_k:=B_k\cap[0,c_k]^d$ with $\emptyset\neq G_k\subseteq\Z^d$ and define
$$D_k:=\big\{w^{(1)}\cdots w^{(k)}:w^{(1)}\in G_1,\cdots,w^{(k)}\in G_k\big\}.$$
Denote the empty word by $\eta$, write $D_0:=\{\eta\}$ and define $D:=\bigcup_{k=0}^\infty D_k$. Let $J_\eta:=J:=[0,1]^d$. For all $k\in\N$ and $w^{(1)}=(w^{(1)}_1,\cdots,w^{(1)}_d)\in G_1$, $\cdots$, $w^{(k)}=(w^{(k)}_1,\cdots,w^{(k)}_d)\in G_k$, define
$$J_{w^{(1)}\cdots w^{(k)}}:=C_1^{-1}w^{(1)}+C_1^{-1}C_2^{-1}w^{(2)}+\cdots+C_1^{-1}\cdots C_k^{-1}w^{(k)}+C_1^{-1}\cdots C_k^{-1}[0,1]^d.$$

Let $\cF:=\{J_w:w\in D\}$,
$$E_k:=\bigcup_{w\in D_k}J_w\quad\text{for all }k\ge0\quad\text{and}\quad E:=\bigcap_{k=0}^\infty E_k.$$

First we prove the following \textbf{Fact 1} and \textbf{Fact 2}.
\vspace{5pt}
\newline\textbf{\underline{Fact 1}.} $\cF$ satisfies the Moran Structure Codition (MSC) defined in \cite[Section 1.2]{HRWW00}.
\begin{itemize}
\item[I.] For any $w\in D$, $J_w$ is obviously geometrically similar to $J$.
\item[II.] For any $k\in\N$ and $w^{(1)}\cdots w^{(k)}\in D_k$, we need to prove $J_{w^{(1)}\cdots w^{(k)}}\subseteq J_{w^{(1)}\cdots w^{(k-1)}}$.
\newline It suffices to show
$$C_1^{-1}w^{(1)}+\cdots+C_1^{-1}\cdots C_{k-1}^{-1}w^{(k-1)}+C_1^{-1}\cdots C_{k-1}^{-1}C_k^{-1}w^{(k)}+C_1^{-1}\cdots C_{k-1}^{-1}C_k^{-1}[0,1]^d$$
$$\subseteq C_1^{-1}w^{(1)}+\cdots+C_1^{-1}\cdots C_{k-1}^{-1}w^{(k-1)}+C_1^{-1}\cdots C_{k-1}^{-1}[0,1]^d,$$
which is equivalent to $C_k^{-1}w^{(k)}+C_k^{-1}[0,1]^d\subseteq[0,1]^d$, and then also equivalent to $w^{(k)}+[0,1]^d\subseteq[0,C_k]^d$. This follows immediately from $w^{(k)}\in G_k\subseteq[0,c_k]^d$ and $c_k+1\le C_k$.
\item[III.] For any $k\ge0$, $w^{(1)}\cdots w^{(k)}\in D_k$ and $u,v\in G_{k+1}$ with $u\neq v$, we need to prove $\text{int}(J_{w^{(1)}\cdots w^{(k)}u})$ $\cap$ $\text{int}(J_{w^{(1)}\cdots w^{(k)}v})=\emptyset$ where $\text{int}(\cdot)$ denotes the interior of a set.
\newline It suffices to show
$$\begin{aligned}
&\big(C_1^{-1}w^{(1)}+\cdots+C_1^{-1}\cdots C_k^{-1}w^{(k)}+C_1^{-1}\cdots C_k^{-1}C_{k+1}^{-1}u+C_1^{-1}\cdots C_k^{-1}C_{k+1}^{-1}(0,1)^d\big)\\
\cap&\big(C_1^{-1}w^{(1)}+\cdots+C_1^{-1}\cdots C_k^{-1}w^{(k)}+C_1^{-1}\cdots C_k^{-1}C_{k+1}^{-1}v+C_1^{-1}\cdots C_k^{-1}C_{k+1}^{-1}(0,1)^d\big)=\emptyset.
\end{aligned}$$
We only need to prove $\big(u+(0,1)^d\big)\cap\big(v+(0,1)^d\big)=\emptyset$. This follows immediately from $u,v\in G_{k+1}\subseteq\Z^d$ and $u\neq v$.
\end{itemize}
\textbf{\underline{Fact 2}.} $\sum_{k=1}^\infty C_1^{-1}\cdots C_k^{-1}G_k=E$.
\newline\boxed{\subset} Let $x\in\sum_{k=1}^\infty C_1^{-1}\cdots C_k^{-1}G_k$. Then there exist $x^{(1)}\in G_1$, $x^{(2)}\in G_2$, $\cdots$ such that $x=\sum_{k=1}^\infty C_1^{-1}\cdots C_k^{-1}x^{(k)}$ converges in $\R^d$. We need to prove $x\in E$. Arbitrarily take an integer $k\ge0$. It suffices to show $x\in\bigcup_{w^{(1)}\cdots w^{(k)}\in D_k}J_{w^{(1)}\cdots w^{(k)}}$. We only need to prove $x\in J_{x^{(1)}\cdots x^{(k)}}$, i.e.,
$$\sum_{n=1}^\infty C_1^{-1}\cdots C_n^{-1}x^{(n)}\in C_1^{-1}x^{(1)}+\cdots+C_1^{-1}\cdots C_k^{-1}x^{(k)}+C_1^{-1}\cdots C_k^{-1}[0,1]^d,$$
which is equivalent to
$$\sum_{n=1}^\infty C_{k+1}^{-1}\cdots C_{k+n}^{-1}x^{(k+n)}\in[0,1]^d,\quad\text{i.e.,}\quad\sum_{n=1}^\infty\frac{(x^{(k+n)}_1,\cdots,x^{(k+n)}_d)}{C_{k+1}\cdots C_{k+n}}\in[0,1]^d.$$
This follows from the fact that for all $j\in\{1,\cdots,d\}$ we have
$$\begin{aligned}
0\le&\sum_{n=1}^\infty\frac{x^{(k+n)}_j}{C_{k+1}\cdots C_{k+n}}\le\sum_{n=1}^\infty\frac{c_{k+n}}{(c_{k+1}+1)\cdots(c_{k+n}+1)}\\
=&\frac{c_{k+1}}{c_{k+1}+1}+\frac{c_{k+2}}{(c_{k+1}+1)(c_{k+2}+1)}+\frac{c_{k+3}}{(c_{k+1}+1)(c_{k+2}+1)(c_{k+3}+1)}+\cdots\\
=&\Big(1-\frac{1}{c_{k+1}+1}\Big)+\Big(\frac{1}{c_{k+1}+1}-\frac{1}{(c_{k+1}+1)(c_{k+2}+1)}\Big)\\
&+\Big(\frac{1}{(c_{k+1}+1)(c_{k+2}+1)}-\frac{1}{(c_{k+1}+1)(c_{k+2}+1)(c_{k+3}+1)}\Big)+\cdots=1.
\end{aligned}$$
\boxed{\supset} Let $x\in E=\bigcap_{k=0}^\infty E_k$. We need to prove $x\in\sum_{k=1}^\infty C_1^{-1}\cdots C_k^{-1}G_k$. By
$$x\in E_1=\bigcup_{w^{(1)}\in G_1}J_{w^{(1)}},$$
there exists $x^{(1)}\in G_1$ such that $x\in J_{x^{(1)}}$. By
$$x\in E_2=\bigcup_{w^{(1)}\in G_1,w^{(2)}\in G_2}J_{w^{(1)}w^{(2)}},$$
there exist $x^{(1)'}\in G_1$ and $x^{(2)}\in G_2$ such that $x\in J_{x^{(1)'}x^{(2)}}$. It follows from $x\in J_{x^{(1)}}\cap J_{x^{(1)'}x^{(2)}}\neq\emptyset$ and the MSC of $\cF$ in \textbf{Fact 1} that $x^{(1)}=x^{(1)'}$. Thus $x\in J_{x^{(1)}x^{(2)}}$.

$\cdots$

Repeating this process we know that there exist $x^{(1)}=(x^{(1)}_1,\cdots,x^{(1)}_d)$ $\in$ $G_1$, $x^{(2)}=(x^{(2)}_1,\cdots,x^{(2)}_d)$ $\in$ $G_2$, $\cdots$ such that $x\in\bigcap_{k=1}^\infty J_{x^{(1)}\cdots x^{(k)}}$. In order to prove $x\in\sum_{k=1}^\infty C_1^{-1}\cdots C_k^{-1}G_k$, we only need to show the following I and II.
\begin{itemize}
\item[I.] Prove that $\sum_{k=1}^\infty C_1^{-1}\cdots C_k^{-1}x^{(k)}$ converges in $\R^d$.
\newline In fact, this follows immediately from
$$0\le\sum_{k=1}^\infty\frac{x^{(k)}_j}{C_1\cdots C_k}\le\sum_{k=1}^\infty\frac{c_k}{(c_1+1)\cdots(c_k+1)}=1\quad\text{for all }j\in\{1,\cdots,d\}.$$
\item[II.] Prove $x=\sum_{k=1}^\infty C_1^{-1}\cdots C_k^{-1}x^{(k)}$.
\newline Since $\lim_{k\to\infty}|J_{x^{(1)}\cdots x^{(k)}}|=0$ and $J_{x^{(1)}}\supseteq J_{x^{(1)}x^{(2)}}\supseteq J_{x^{(1)}x^{(2)}x^{(3)}}\supseteq\cdots$ are all closed sets, we get $\#(\bigcap_{k=1}^\infty J_{x^{(1)}\cdots x^{(k)}})=1$. In order to prove $x=\sum_{k=1}^\infty C_1^{-1}\cdots C_k^{-1}x^{(k)}$, by $x\in\bigcap_{k=1}^\infty J_{x^{(1)}\cdots x^{(k)}}$, it suffices to show $\sum_{k=1}^\infty C_1^{-1}\cdots C_k^{-1}x^{(k)}\in\bigcap_{k=1}^\infty J_{x^{(1)}\cdots x^{(k)}}$. In fact this follows in the same way as in the proof of the above ``\boxed{\subset}''.
\end{itemize}

Now we deduce statements (1) and (2) in Corollary \ref{dim-sum} from Theorem \ref{sum union}. For all $k\in\N$, let $A_k:=C_1^{-1}\cdots C_k^{-1}G_k\neq\emptyset$ and $A_k':=C_1^{-1}\cdots C_k^{-1}(B_k\setminus[0,c_k]^d)$ (may be $\emptyset$).
\vspace{5pt}
\newline(1) Suppose that $B_1,B_2,\cdots$ are all closed, (\ref{cor inf-sup or}) holds for every $j\in\{1,\cdots,d\}$, and
$$\lim_{k\to\infty}\frac{\inf\{|x|:x\in B_k\setminus[0,c_k]^d\}}{C_1\cdots C_k}=+\infty,\quad\text{i.e.,}\quad\lim_{k\to\infty}\inf_{a\in A_k'}|a|=+\infty.$$
We need to prove that $\sum_{k=1}^\infty(A_k\cup A_k')$ is closed. For every $j\in\{1,\cdots,d\}$, since (\ref{cor inf-sup or}) implies (\ref{inf-sup or}), by Theorem \ref{sum union} (1), it suffices to verify (\ref{inf-sup con}). In fact this follows immediately from
$$\max\Big\{\sum_{k=1}^\infty\big|\min(A_k)_j\big|,\sum_{k=1}^\infty\big|\max(A_k)_j\big|\Big\}\le\sum_{k=1}^\infty\frac{c_k}{C_1\cdots C_k}\le\sum_{k=1}^\infty\frac{c_k}{(c_1+1)\cdots(c_k+1)}=1.$$
(2) Suppose that $B_k$ is at most countable for every $k\in\N$ and
$$\varliminf_{k\to\infty}\frac{\inf\{|x|:x\in B_k\setminus[0,c_k]^d\}}{C_1\cdots C_k}>0,\quad\text{i.e.,}\quad\varliminf_{k\to\infty}\inf_{a\in A_k'}|a|>0.$$
\begin{itemize}
\item[\textcircled{\footnotesize{$1$}}] Suppose $\prod_{k=1}^\infty\frac{\#G_k}{C_k^d}=0$ and we need to prove $\cL^d(\sum_{k=1}^\infty(A_k\cup A_k'))=0$.
\newline Since $\cL^d$ and $\cH^d$ are equivalent, by Theorem \ref{sum union} (2) \textcircled{\footnotesize{$1$}}, it suffices to show $\cL^d(\sum_{k=1}^\infty A_k)$ $=0$. Recalling \textbf{Fact 2}, we only need to prove $\cL^d(E)=0$. In fact, this follows immediately from
$$\cL^d(E)\le\cL^d(E_k)\le\sum_{w\in D_k}\cL^d(J_w)=\sum_{w\in D_k}\cL^d(C_1^{-1}\cdots C_k^{-1}[0,1]^d)=\frac{\#G_1\cdots\#G_k}{C_1^d\cdots C_k^d}\to0$$
as $k\to\infty$ using $\prod_{k=1}^\infty\frac{\#G_k}{C_k^d}=0$.
\item[\textcircled{\footnotesize{$2$}}] Suppose $\lim_{k\to\infty}\frac{\log C_k}{\log C_1\cdots C_k}=0$ and we need to prove the Hausdorff and packing dimension formulae for $\sum_{k=1}^\infty C_1^{-1}\cdots C_k^{-1}B_k$, which is equal to $\sum_{k=1}^\infty(A_k\cup A_k')$. Since the proofs of the two formulae are similar, in the following we only prove the Hausdorff one, i.e.,
    $$\dim_H\sum_{k=1}^\infty(A_k\cup A_k')=\varliminf_{k\to\infty}\frac{\log\#G_1\cdots\#G_k}{\log C_1\cdots C_k}.$$
    By Theorem \ref{sum union} (2) \textcircled{\footnotesize{$1$}} and \textbf{Fact 2}, it suffices to show
    $$\dim_HE=\varliminf_{k\to\infty}\frac{\log\#G_1\cdots\#G_k}{\log C_1\cdots C_k}$$
    using \cite[Theorem 1.3]{HRWW00}. Since
$$0=\lim_{k\to\infty}\frac{\log C_k}{\log C_1\cdots C_k}\le\lim_{k\to\infty}\frac{\log C_k}{\log C_1\cdots C_k-\log\sqrt{d}}\le\lim_{k\to\infty}\frac{\log C_k}{\frac{1}{2}\log C_1\cdots C_k}=0,$$
we get
$$\lim_{k\to\infty}\frac{\log\frac{1}{C_k}}{\log\max_{w\in D_k}|J_w|}=\lim_{k\to\infty}\frac{-\log C_k}{\log|C_1^{-1}\cdots C_k^{-1}[0,1]^d|}=\lim_{k\to\infty}\frac{\log C_k}{\log C_1\cdots C_k-\log\sqrt{d}}=0.$$
It follows from \textbf{Fact 1} and \cite[Theorem 1.3]{HRWW00} that
$$\dim_HE=\varliminf_{k\to\infty}\frac{\log\#G_1\cdots\#G_k}{-\log\frac{1}{C_1}\cdots\frac{1}{C_k}}=\varliminf_{k\to\infty}\frac{\log\#G_1\cdots\#G_k}{\log C_1\cdots C_k}.$$
\end{itemize}
\end{proof}

\section{Proofs of Corollaries \ref{inter-com} and \ref{inter-non}}
\indent

Using Corollary \ref{dim-spt} and Theorem \ref{SC}, we can deduce Corollaries \ref{inter-com} and \ref{inter-non} by constructing sequences $\{m_k\}_{k\ge1}$ and $\{N_k\}_{k\ge1}$ similar to the $\{b_k\}_{k\ge1}$ and $\{N_k\}_{k\ge1}$ given in the proof of \cite[Theorem 1.7]{LMW22}. For self-contained and for the convenience of the readers, we still give the detailed proofs as follows.

\begin{proof}[Proof of Corollary \ref{inter-com}] Let $d\in\N$ and arbitrarily take $\alpha,\beta\in[0,d]$ with $\alpha\le\beta$. Let $m_1=2$ and $m_k=k^2$ for all $k\ge2$. Define a family of functions $g_\gamma:\N\to\N$ for $\gamma\in[0,1]$ by
$$g_\gamma(n):=\left\{\begin{array}{ll}
n^{1+\lfloor\log n\rfloor} & \text{if }\gamma=0,\\
\lfloor n^{\frac{1}{\gamma}-1}\rfloor n & \text{if }0<\gamma<1,\\
2n & \text{if }\gamma=1,
\end{array}\right.$$
where $\lfloor x\rfloor$ denotes the integer part of $x$. Then
\begin{equation}\label{lim to gamma-compact}
\lim_{n\to\infty}\frac{\log n}{\log g_\gamma(n)}=\gamma\quad\text{for all }\gamma\in[0,1].
\end{equation}
Choose a strictly increasing sequence of integers $\{l_j\}_{j=1}^\infty$ such that $l_1=0$ and
\begin{equation}\label{l lim-compact}
\lim_{j\to\infty}\frac{l_j\log l_j}{l_{j+1}-l_j}=0.
\end{equation}
For any $k\in\N$, let
$$N_k:=\left\{\begin{array}{ll}
g_\frac{\alpha}{d}(m_k) & \text{if } l_j<k\le l_{j+1}\text{ for some odd }j\in\N,\\
g_\frac{\beta}{d}(m_k) & \text{if } l_j<k\le l_{j+1}\text{ for some even }j\in\N,
\end{array}\right.$$
and let $B_k:=\{0,1,\cdots,m_k-1\}^d$. Then $N_k\ge m_k\ge2$ are integers with $m_k\mid N_k$ for all $k\in\N$, the condition (\ref{cor lim infty}) holds where $\min\emptyset$ is regarded as $+\infty$, $\prod_{k=1}^\infty\frac{\#B_k}{N_k^d}=\prod_{k=1}^\infty\frac{m_k^d}{N_k^d}=0$ by $N_k\ge2m_k$ for all $k$ large enough, and $\{B_k\}_{k\ge1}$ is a sequence of nearly $d$-th power lattices with respect to $\{m_k\}_{k\ge1}$ and the sequence of $d\times d$ diagonal matrices $\{\text{diag}(N_k,\cdots,N_k)\}_{k\ge1}$. In order to use Corollary \ref{dim-spt}, it suffices to prove $\lim_{k\to\infty}\frac{\log N_k}{\log N_1\cdots N_k}=0$. In fact, by $2m_k\le N_k\le m_k^{1+\log m_k}$ for all $k\in\N$ large enough, we have
$$\begin{aligned}
\lim_{k\to\infty}\frac{\log N_k}{\log N_1N_2\cdots N_k}&\le\lim_{k\to\infty}\frac{\log m_k^{1+\log m_k}}{\log 2m_12m_2\cdots2m_k}=\lim_{k\to\infty}\frac{(1+\log m_k)\log m_k}{k\log2+\log m_1m_2\cdots m_k}\\
&=\lim_{k\to\infty}\frac{(1+\log k^2)\log k^2}{k\log2+\log2\cdot2^2\cdot3^2\cdots k^2}\le\lim_{k\to\infty}\frac{4(\log k)^2+2\log k}{k\log2}\\
&=\frac{4}{\log2}\lim_{k\to\infty}\frac{(\log k)^2}{k}+\frac{2}{\log2}\lim_{k\to\infty}\frac{\log k}{k}=0.
\end{aligned}$$
Therefore, by applying Corollary \ref{dim-spt}, we know that the infinite convolution
$$\mu=\delta_{N_1^{-1}B_1}*\delta_{N_1^{-1}N_2^{-1}B_2}*\delta_{N_1^{-1}N_2^{-1}N_3^{-1}B_3}*\cdots$$
exists, is a singular spectral measure with a spectrum in $\Z^d$, $\spt\mu=\sum_{k=1}^\infty N_1^{-1}\cdots N_k^{-1}B_k$,
$$\dim_H\spt\mu=\varliminf_{k\to\infty}\frac{d\log m_1\cdots m_k}{\log N_1\cdots N_k}\quad\text{and}\quad\dim_P\spt\mu=\varlimsup_{k\to\infty}\frac{d\log m_1\cdots m_k}{\log N_1\cdots N_k}.$$
To complete the proof, it suffices to show the following (1), (2) and (3).
\begin{itemize}
\item[(1)] Prove that $\spt\mu$ is compact.
\newline In fact this follows immediately from Corollary \ref{spt-cor} (2) and for all $j\in\{1,\cdots,d\}$,
$$\sum_{k=1}^\infty\big|\min(N_1^{-1}\cdots N_k^{-1}B_k)_j\big|=0<\infty$$
and
$$\sum_{k=1}^\infty\big|\max(N_1^{-1}\cdots N_k^{-1}B_k)_j\big|=\sum_{k=1}^\infty\frac{m_k-1}{N_1\cdots N_k}\le\sum_{k=1}^\infty\frac{1}{N_1\cdots N_{k-1}}\cdot\frac{m_k}{N_k}\le\sum_{k=1}^\infty\frac{1}{2^{k-1}}=2<\infty.$$
\item[(2)] Prove $\varliminf_{k\to\infty}\frac{\log m_1m_2\cdots m_k}{\log N_1N_2\cdots N_k}=\frac{\alpha}{d}$.

On the one hand, we have
$$\varliminf_{k\to\infty}\frac{\log m_1m_2\cdots m_k}{\log N_1N_2\cdots N_k}\overset{(\star)}{\ge}\varliminf_{k\to\infty}\frac{\log m_k}{\log N_k}\overset{(\star\star)}{\ge}\varliminf_{k\to\infty}\frac{\log m_k}{\log g_\frac{\alpha}{d}(m_k)}\xlongequal[\text{(\ref{lim to gamma-compact})}]{\text{by}}\frac{\alpha}{d},$$
where ($\star$) follows from Theorem \ref{SC} and ($\star\star$) follows from $g_\frac{\beta}{d}(m_k)\le g_\frac{\alpha}{d}(m_k)$ for all $k$ large enough with $0\le\frac{\alpha}{d}\le\frac{\beta}{d}\le1$.

On the other hand, we have
\begin{align*}
\varliminf_{k\to\infty}&\frac{\log m_1m_2\cdots m_k}{\log N_1N_2\cdots N_k}\le\varliminf_{j\to\infty}\frac{\log m_1m_2\cdots m_{l_{2j}}}{\log N_1N_2\cdots N_{l_{2j}}}\\
&\le\varliminf_{j\to\infty}\frac{\log m_1m_2\cdots m_{l_{2j-1}}+\log m_{l_{2j-1}+1}m_{l_{2j-1}+2}\cdots m_{l_{2j}}}{\log N_{l_{2j-1}+1}N_{l_{2j-1}+2}\cdots N_{l_{2j}}}\\
&\le\varliminf_{j\to\infty}\Big(\frac{l_{2j-1}\log m_{l_{2j-1}}}{(l_{2j}-l_{2j-1})\log g_\frac{\alpha}{d}(m_{l_{2j-1}+1})}+\frac{\log m_{l_{2j-1}+1}m_{l_{2j-1}+2}\cdots m_{l_{2j}}}{\log N_{l_{2j-1}+1}N_{l_{2j-1}+2}\cdots N_{l_{2j}}}\Big)\\
&\le\varliminf_{j\to\infty}\Big(\frac{l_{2j-1}}{l_{2j}-l_{2j-1}}\cdot\frac{\log m_{l_{2j-1}+1}}{\log g_\frac{\alpha}{d}(m_{l_{2j-1}+1})}+\frac{\log m_{l_{2j-1}+1}+\log m_{l_{2j-1}+2}+\cdots+\log m_{l_{2j}}}{\log N_{l_{2j-1}+1}+\log N_{l_{2j-1}+2}+\cdots+\log N_{l_{2j}}}\Big)\\
&\xlongequal[\text{and }(\ref{lim to gamma-compact})]{\text{by }(\ref{l lim-compact})}\varliminf_{j\to\infty}\frac{\log m_{l_{2j-1}+1}+\log m_{l_{2j-1}+2}+\cdots+\log m_{l_{2j}}}{\log g_\frac{\alpha}{d}(m_{l_{2j-1}+1})+\log g_\frac{\alpha}{d}(m_{l_{2j-1}+2})+\cdots+\log g_\frac{\alpha}{d}(m_{l_{2j}})}=\frac{\alpha}{d},
\end{align*}
where the last equality can be proved as follows. Let $r:=\frac{\alpha}{d}\in[0,1]$, and for all $n\in\N$ let
$$a_n=\log m_1+\log m_2+\cdots+\log m_{l_{2n}},\quad c_n=\log m_1+\log m_2+\cdots+\log m_{l_{2n-1}},$$
$$b_n=\log g_r(m_1)+\log g_r(m_2)+\cdots+\log g_r(m_{l_{2n}}), d_n=\log g_r(m_1)+\log g_r(m_2)+\cdots+\log g_r(m_{l_{2n-1}}).$$
It suffices to prove $\lim_{n\to\infty}\frac{a_n-c_n}{b_n-d_n}=r$. Since Theorem \ref{SC} and (\ref{lim to gamma-compact}) imply $\lim_{n\to\infty}\frac{a_n}{b_n}=\lim_{n\to\infty}\frac{c_n}{d_n}=r$, by Proposition \ref{diff-quot} we only need to verify $\varliminf_{n\to\infty}\frac{b_n}{d_n}>1$. In fact this follows immediately from
$$\begin{aligned}
\varliminf_{n\to\infty}\frac{b_n}{d_n}-1&=\varliminf_{n\to\infty}\frac{\log g_r(m_{l_{2n-1}+1})+\log g_r(m_{l_{2n-1}+2})+\cdots+\log g_r(m_{l_{2n}})}{\log g_r(m_1)+\log g_r(m_2)+\cdots+\log g_r(m_{l_{2n-1}})}\\
&\ge\varliminf_{n\to\infty}\frac{(l_{2n}-l_{2n-1})\cdot\log g_r(m_{l_{2n-1}})}{l_{2n-1}\cdot\log g_r(m_{l_{2n-1}})}\xlongequal[\text{(\ref{l lim-compact})}]{\text{by}}\infty.
\end{aligned}$$
\item[(3)] Prove $\varlimsup_{k\to\infty}\frac{\log m_1m_2\cdots m_k}{\log N_1N_2\cdots N_k}=\frac{\beta}{d}$.

On the one hand, we have
$$\varlimsup_{k\to\infty}\frac{\log m_1m_2\cdots m_k}{\log N_1N_2\cdots N_k}\overset{(\star)}{\le}\varlimsup_{k\to\infty}\frac{\log m_k}{\log N_k}\overset{(\star\star)}{\le}\varlimsup_{k\to\infty}\frac{\log m_k}{\log g_\frac{\beta}{d}(m_k)}\xlongequal[\text{(\ref{lim to gamma-compact})}]{\text{by}}\frac{\beta}{d},$$
where ($\star$) follows from Theorem \ref{SC} and ($\star\star$) follows from $g_\frac{\alpha}{d}(m_k)\ge g_\frac{\beta}{d}(m_k)$ for all $k$ large enough with $0\le\frac{\alpha}{d}\le\frac{\beta}{d}\le1$.

On the other hand, we have
$$\begin{aligned}
\varlimsup_{k\to\infty}\frac{\log m_1m_2\cdots m_k}{\log N_1N_2\cdots N_k}&\ge\varlimsup_{j\to\infty}\frac{\log m_1m_2\cdots m_{l_{2j+1}}}{\log N_1N_2\cdots N_{l_{2j+1}}}\\
&\ge\varlimsup_{j\to\infty}\frac{\log m_{l_{2j}+1}m_{l_{2j}+2}\cdots m_{l_{2j+1}}}{\log N_1N_2\cdots N_{l_{2j}}+\log N_{l_{2j}+1}N_{l_{2j}+2}\cdots N_{l_{2j+1}}}\\
&\ge\varlimsup_{j\to\infty}\frac{\log m_{l_{2j}+1}m_{l_{2j}+2}\cdots m_{l_{2j+1}}}{l_{2j}\log g_\frac{\alpha}{d}(m_{l_{2j}})+\log g_\frac{\beta}{d}(m_{l_{2j}+1})g_\frac{\beta}{d}(m_{l_{2j}+2})\cdots g_\frac{\beta}{d}(m_{l_{2j+1}})}\\
&\overset{(\star)}{=}\varlimsup_{j\to\infty}\frac{\log m_{l_{2j}+1}m_{l_{2j}+2}\cdots m_{l_{2j+1}}}{\log g_\frac{\beta}{d}(m_{l_{2j}+1})g_\frac{\beta}{d}(m_{l_{2j}+2})\cdots g_\frac{\beta}{d}(m_{l_{2j+1}})}=\frac{\beta}{d},
\end{aligned}$$
where the last equality can be proved in the same way as the end of the above (2), and ($\star$) follows from
\newpage
\begin{align*}
\lim_{j\to\infty}\frac{l_{2j}\log g_\frac{\alpha}{d}(m_{l_{2j}})}{\log m_{l_{2j}+1}m_{l_{2j}+2}\cdots m_{l_{2j+1}}}&\le\lim_{j\to\infty}\frac{l_{2j}(1+\log m_{l_{2j}})\log m_{l_{2j}}}{(l_{2j+1}-l_{2j})\log m_{l_{2j}+1}}\\
&\le\lim_{j\to\infty}\frac{l_{2j}(1+\log l_{2j}^2)}{l_{2j+1}-l_{2j}}\xlongequal[(\ref{l lim-compact})]{\text{by}}0,
\end{align*}
where the first inequality follows from $g_\frac{\alpha}{d}(n)\le n^{1+\log n}$ for all $n\in\N$ large enough.
\end{itemize}
\end{proof}

\begin{proof}[Proof of Corollary \ref{inter-non}] Let $d\in\N$ and arbitrarily take $\alpha,\beta\in[0,d]$ with $\alpha\le\beta$. Let $\{m_k\}_{k\ge1}$ and $\{N_k\}_{k\ge1}$ be defined as in the proof of Corollary \ref{inter-com}. For all $k\in\N$, let
$$B_k:=\{0,1,\cdots,m_k-2,N_1\cdots N_k\cdot k+m_k-1\}^d$$
and
$$G_k:=B_k\cap\{0,1,\cdots,m_k-1\}^d=\{0,1,\cdots,m_k-2\}^d.$$
In a way similar to the proof of Corollary \ref{inter-com}, by applying Corollary \ref{dim-spt}, we know that the infinite convolution
$$\mu=\delta_{N_1^{-1}B_1}*\delta_{N_1^{-1}N_2^{-1}B_2}*\delta_{N_1^{-1}N_2^{-1}N_3^{-1}B_3}*\cdots$$
exists, is a singular spectral measure with a spectrum in $\Z^d$, $\spt\mu=\sum_{k=1}^\infty N_1^{-1}\cdots N_k^{-1}B_k$,
$$\dim_H\spt\mu=\varliminf_{k\to\infty}\frac{d\log(m_1-1)\cdots(m_k-1)}{\log N_1\cdots N_k}\text{ and }\dim_P\spt\mu=\varlimsup_{k\to\infty}\frac{d\log(m_1-1)\cdots(m_k-1)}{\log N_1\cdots N_k}.$$
To complete the proof, we only need to show the following (1) and (2).
\begin{itemize}
\item[(1)] Prove that $\spt\mu$ is not compact.
\newline In fact this follows immediately from Corollary \ref{spt-cor} (2) and
$$\sum_{k=1}^\infty\max(N_1^{-1}\cdots N_k^{-1}B_k)_1=\sum_{k=1}^\infty\frac{N_1\cdots N_k\cdot k+m_k-1}{N_1\cdots N_k}\ge\sum_{k=1}^\infty k=\infty.$$
\item[(2)] Prove $\varliminf_{k\to\infty}\frac{\log(m_1-1)\cdots(m_k-1)}{\log N_1\cdots N_k}=\frac{\alpha}{d}$ and $\varlimsup_{k\to\infty}\frac{\log(m_1-1)\cdots(m_k-1)}{\log N_1\cdots N_k}=\frac{\beta}{d}$.
\newline By (2) and (3) in the proof of Corollary \ref{inter-com}, we only need to show $\lim_{k\to\infty}\frac{\log(m_1-1)\cdots(m_k-1)}{\log m_1\cdots m_k}$ $=$ $1$. In fact, it follows from
$$\lim_{k\to\infty}\frac{\log(m_k-1)}{\log m_k}=\lim_{k\to\infty}\frac{\log(k^2-1)}{\log k^2}=1$$
that
$$\begin{aligned}
\lim_{k\to\infty}\frac{\log(m_1-1)\cdots(m_k-1)}{\log m_1\cdots m_k}&=\lim_{k\to\infty}\frac{\log(m_1-1)+\cdots+\log(m_k-1)}{\log m_1+\cdots+\log m_k}\\
&\xlongequal[\text{Theorem \ref{SC}}]{\text{by}}\lim_{k\to\infty}\frac{\log(m_k-1)}{\log m_k}=1.
\end{aligned}$$
\end{itemize}
\end{proof}

\begin{ack}
This work was supported by ``National Natural Science Foundation of China'' (NSFC 12201652) and (NSFC 12271534).
\end{ack}

$ $
\newline\noindent Address: School of Mathematics and Statistics,
\newline$\text{ }\quad\text{ }\quad\text{ }\quad\text{ }$ Guangdong University of Technology,
\newline$\text{ }\quad\text{ }\quad\text{ }\quad\text{ }$ Guangzhou, 510520, P.R. China
\newline$ $
\newline Emails: yaoqiang.li@gdut.edu.cn
\newline$\text{ }\quad\text{ }\quad\text{ }\text{ }\text{ }$ scutyaoqiangli@qq.com

\end{document}